\documentclass[12pt,a4paper]{article}
\usepackage[utf8]{inputenc}
\usepackage{amsmath}
\usepackage{amsfonts}
\usepackage{amssymb}
\usepackage[left=2cm,right=2cm,top=2cm,bottom=2cm]{geometry}
\usepackage{amsfonts,amsthm}
\usepackage{graphicx}
\usepackage{epstopdf}
\usepackage{algorithm}
\usepackage{algorithmic}
\usepackage{amsmath}
\usepackage{dsfont}
\usepackage{amssymb}
\usepackage{bm}
\usepackage{makecell}
\usepackage{float}
\usepackage{mathtools}
\usepackage{psfrag}
\usepackage{multicol}
\usepackage{subcaption}
\usepackage{colortbl}
\usepackage{pgfplots}
\usepackage{tikz}

\newtheorem{theorem}{Theorem}[section]
\newtheorem{lemma}[theorem]{Lemma}

\newtheorem{remark}[theorem]{Remark}
\usetikzlibrary{external}
\ifpdf
  \DeclareGraphicsExtensions{.eps,.pdf,.png,.jpg}
\else
  \DeclareGraphicsExtensions{.eps}
\fi

\usepackage{hyperref}
\newcommand{\footremember}[2]{%
	\footnote{#2}
	\newcounter{#1}
	\setcounter{#1}{\value{footnote}}%
}

\author{%
	Cécile Haberstich\footremember{cea}{CEA, DAM, DIF, F-91297 Arpajon France}%
	\and Anthony Nouy\footremember{centrale}{Centrale Nantes, LMJL (UMR CNRS 6629)} %
	\and Guillaume Perrin\footremember{gustaveeiffel}{COSYS, Université Gustave Eiffel, 77420 Champs-sur-Marne, France} %
}
\usepackage{cleveref}

\title{Active learning of tree tensor networks using optimal least-squares}
\date{}

\begin{document}

\maketitle

\begin{abstract}
In this paper, we propose new learning algorithms for approximating high-dimensional functions using tree tensor networks in a least-squares setting. Given a dimension tree or architecture of the tensor network, we provide an algorithm that generates a sequence of nested tensor subspaces based on a generalization of principal component analysis for multivariate functions. An optimal least-squares method is used for computing projections onto the generated tensor subspaces, using samples generated from a distribution depending on the previously generated subspaces. 
We provide an error bound in expectation for the obtained approximation. 
Practical strategies are proposed for adapting the feature spaces and ranks to achieve a prescribed error.
Also, we propose an algorithm that progressively constructs the dimension tree by suitable pairings of variables, that allows 
to further reduce the number of samples necessary to reach that error. Numerical examples illustrate the performance of the proposed algorithms and show that stable approximations are obtained with a number of samples close to the number of free parameters of the estimated tensor networks. 
\end{abstract}

\begin{keywords}
high-dimensional approximation, tree-based tensor formats, principal component analysis, adaptive strategies, active learning with weighted least-squares
\end{keywords}


	\section{Introduction}
~
The approximation of high-dimensional functions raises many challenges. Especially for uncertainty quantification problems where a function represents a model depending on a potentially high number of parameters. Such problems require many evaluations of the functions which is intractable when the model is costly to evaluate. A solution is then to construct a surrogate model which amounts in approximating the relation between an output random variable $Y$ and input random variables representing the parameters $X = (X_1, \hdots, X_d)$ using samples 
of $(X,Y)$.\\

When the dimension $d$ is high, using approximation tools adapted to standard regularity classes (e.g. splines for isotropic Sobolev or Besov regularity) leads to a complexity of the approximation methods which grows exponentially with the dimension $d$. This is the so-called curse of dimensionality. To expect a good approximation in a context where the number of evaluations of a function has to be moderate, we have to assume that the functions present some low-dimensional structures. Exploiting these structures of the function usually requires particular approximation tools, which may be application dependent. 
~
An approximation tool that achieve good performances for many classes of functions is the class of tree tensor networks or functions in tree-based tensor format. Given a partition tree $T$ over $D := \{1, \hdots, d\}$ and a tuple $r = (r_\alpha)_{\alpha \in D}$ of integers, a tree based tensor format $\mathcal{T}_r^T(V)$ is defined as the set of functions in some finite-dimensional tensor space $V$ (or feature tensor space) whose $\alpha$-ranks are bounded by $r_\alpha$. A function $u \in \mathcal{T}_r^T(V)$ therefore admits for each $\alpha \in T$ a 
finite-rank representation 
\begin{equation}
\label{eq:alpha_rank}
v(x) = \sum^{r_{\alpha}}_{i=1}v_i^{\alpha}(x_{\alpha})w_i^{\alpha^c}(x_{\alpha^c}),
\end{equation}
where the $v_i^{\alpha}$ and $w_i^{\alpha^c}$ are functions of complementary groups of variables.
It admits a multilinear parametrization with parameters forming a tree network of low-order tensors, hence the name tree tensor networks. Also, it has been identified with a particular class of deep neural networks (more precisely sum-product networks or arithmetic circuits) \cite{CohenSharirShashua2016}. For a detailed introduction to tree tensor networks, we refer the reader to the monograph \cite{Hackbusch2015} and surveys \cite{BachmayrSchneiderUschmajew2016,FalcoNouy2018,Nouy2017b1,cichocki2016tensor1}. 
\\

Several algorithms for constructing approximations in tree-based tensor formats using points evaluations of functions have already been proposed. On the one hand, there are learning approaches that use random and unstructured evaluations of the functions \cite{StoudenmireSchwab2016,GrelierNouyChevreuil2018,HaMiLiu2020}. 
These algorithms are yet mainly based on heuristics and lack of theoretical guarantees. On the other hand, there are (fewer) algorithms 
that use adaptive and structured evaluations of functions. Among them, we can distinguish extensions of (adaptive) cross approximation to higher-order tensor (see \cite{LuuMadayGuilloGuerin2017} for the Tucker format, or \cite{OseledetsTyrtyshnikov2010} and \cite{BallaniGrasedyckKluge2013} for tree-based tensor formats) from extensions of the singular value decomposition (SVD) to higher-order tensors (see \cite{Lathauwer2000}, \cite{Grasedyck2010} and \cite{Nouy2017}). Among higher-order singular value decomposition (HOSVD) approaches, the method from \cite{Nouy2017} is of particular interest, the principle is to construct a hierarchy of optimal subspaces that results in a final tensor product space in which the function $u$ is projected. Under strong assumptions on the estimation error made in the determination of subspaces, the author in \cite{Nouy2017} shows that with a number of evaluations scaling as the complexity (i.e. the number of parameters) of the tree-based tensor format, the approximation is quasi-optimal but with constants depending on some projection operators which are not properly quantified. 
Devising learning algorithms coming with theoretical guarantees remains an open challenge. 
\\

In this work we propose an algorithm adapted from \cite{Nouy2017} that constructs an approximation of $u$ in tree-based tensor format, using adaptive and structured sampling, with near-optimality results under some assumptions on the function and the number of samples. Also we propose heuristic strategies for obtaining an approximation with a desired precision and near-optimal complexity. Given a tree $T$, and using a leaves-to-root approach, the algorithm constructs, thanks to a series of principal component analyses (PCA), low-dimensional subspaces of functions of groups of variables associated with each node of the tree. More precisely, for each node of the tree $\alpha \in T \setminus \{D\}$, we construct the $\alpha$-principal subspace of an oblique projection of $u$ (that is to say an approximation of the $\alpha$-principal subspace of $u$).
For the projection, we use the boosted optimal weighted least-squares projection \cite{HaNoPe2019}. Using this strategy the error has several contributions: a discretization error (due to the use of a finite-dimensional feature space $V$), a truncation error (due to the finite ranks $r$) and an estimation error (due to the limited number of samples). We propose a (partially) heuristic adaptive algorithm that controls simultaneously the discretization, truncation and estimation errors.\\

The above algorithm works for an arbitrary but fixed dimension partition tree $T$. However the ranks and therefore the number of evaluations $n$ necessary to reach a given precision may strongly depend on the chosen tree $T$. Choosing the tree which minimizes the number of evaluations $n$ for a given accuracy is a
combinatorial optimization problem, that is intractable in practice. In \cite{GrelierNouyChevreuil2018} and \cite{GrelierNouyLebrun2019}, the authors propose a stochastic algorithm that explores a reasonable number of dimension trees with the same arity. The key idea is to favour the exploration of trees yielding low ranks for a given precision. In \cite{BallaniGrasedyck2014}, the authors propose a deterministic strategy that constructs a dimension tree in a leaves-to-root approach by successive pairing of nodes. The pairings are chosen in order to minimize a certain cost functional based on estimated $\alpha$-ranks. The selected tree can be used to compute the approximation of $u$. The number of function's evaluations used to estimate the $\alpha$-ranks adds up to the number of evaluations necessary to compute the approximation. In this paper, we propose a new approach that progressively constructs a dimension tree by suitable pairings of variables (using stochastic optimization) and estimate the principal subspaces associated with the newly selected nodes.\\
%

The outline of the paper is as follows. In Section \ref{sec:pca_multivariate_functions}, we first present the notion of principle component analysis for multivariate functions with the definition of $\alpha$-principal subspaces. We then propose a strategy to estimate these spaces relying on an approximation with a particular oblique projection and an adaptive statistical estimation. 
In Section \ref{sec:learning_ttnetworks_PCA}, we present and analyze the algorithm for learning a tree tensor network given a fixed dimension tree. In Section \ref{sec:tree_adaptation}, we present the constructive approach for selecting a dimension tree. Finally, Section \ref{sec:numerical_examples} demonstrates the efficiency of the proposed algorithms on numerical examples.

\section{Principal component analysis of multivariate functions}\label{sec:pca_multivariate_functions}

For $d>1$, let $\mathcal{X} = \mathcal{X}_1 \times \hdots \times \mathcal{X}_d$ be a product set in $\mathbb{R}^d$, and $\mu = \mu_1 \otimes \hdots \otimes \mu_d$ be a product measure on $\mathcal{X}$. The Hilbert space of real-valued square-integrable functions defined on $\mathcal{X}$ is denoted by $L^2_{\mu}(\mathcal{X})$. Let $\vert \vert \cdot \vert \vert_{L^2_{\mu}(\mathcal{X})}$ be the natural norm in $L^2_{\mu}(\mathcal{X})$, defined by
\begin{equation}
\Vert v \Vert_{L^2_{\mu}}^2 = \int_{\mathcal{X}} v(x)^2d\mu(x), \ v \in L^2_{\mu}(\mathcal{X}).
\end{equation}
For each $x=(x_1,\ldots,x_d)\in\mathcal{X}$ and each non-empty subset $\alpha$ of $D$, we write $x_{\alpha} = (x_{\nu})_{\nu \in \alpha}$, $\mu_{\alpha} = \otimes_{\nu \in \alpha} \mu_{\nu}$, and $\mathcal{X}_{\alpha} = \bigtimes_{\nu \in \alpha} \mathcal{X}_{\nu}$. Up to a reordering of the variables $x_1, \hdots, x_d$, a function $u$ defined on $\mathcal{X}$ can be identified with a bivariate function $u(x_{\alpha}, x_{\alpha^c})$ defined in $\mathcal{X}_{\alpha} \times \mathcal{X}_{\alpha^c}$, where $\alpha^c = D \setminus \alpha$.\\

~ The $\alpha$-rank of $u$, denoted by $\text{rank}_{\alpha}(u)$, is the canonical rank of $u(x_{\alpha}, x_{\alpha^c})$, that is the minimal integer such that for some functions $f_k^{\alpha} \in L^2_{\mu_{\alpha}}(\mathcal{X}_{\alpha})$, and $ f_k^{\alpha^c} \in L^2_{\mu_{\alpha^c}}(\mathcal{X}_{\alpha^c})$
\begin{equation}
\label{def:matricisation}
u(x) = \sum_{k=1}^{\text{rank}_{\alpha}(u)}f_k^{\alpha}(x_{\alpha})f_k^{\alpha^c}(x_{\alpha^c}).
\end{equation}

For a $m_\alpha$-dimensional subspace $V_{\alpha} \subset L^2_{\mu_{\alpha}}(\mathcal{X}_{\alpha})$, we denote by $P_{V_{\alpha}}$ the orthogonal projection from $L^2_{\mu_{\alpha}}(\mathcal{X}_{\alpha})$ to $V_{\alpha}$, and by $\mathcal{P}_{U_{\alpha}}$ the orthogonal projection from $L^2_{\mu}(\mathcal{X})$ to $V_{\alpha} \otimes L^2_{\mu_{\alpha^c}}(\mathcal{X}_{\alpha^c})$, such that for all $x_{\alpha^c} \in \mathcal{X}_{\alpha^c}$, $(\mathcal{P}_{U_{\alpha}}u)(\cdot, x_{\alpha^c}) = P_{U_{\alpha}} u(\cdot, x_{\alpha^c})$.\\

~From now on, for the sake of clarity and when there is no ambiguity, we will denote $L^2_{\mu}: = L^2_{\mu}(\mathcal{X})$, the norm $\Vert v \Vert : = \Vert v \Vert_{L^2_{\mu}(\mathcal{X})}$ and the associated inner product $(\cdot , \cdot) : = (\cdot , \cdot)_{L^2_{\mu}(\mathcal{X})} $. Also, we let $L^2_{\mu_{\alpha}} := L^2_{\mu_{\alpha}}(\mathcal{X}_{\alpha})$.
\subsection{$\alpha$-principal subspaces}\label{sec:alpha_principal_subspaces}

Let $\alpha \subset D$ and $u$ be a function in $L^2_{\mu}$ with $\text{rank}_{\alpha}(u)\in\mathbb{N}\cup \{+\infty\}$. 
For each $\alpha \subset D$, the function $u$ admits the following singular value decomposition
\begin{equation}
\label{eq:svd_multi_function}
u(x) = \sum_{k=1}^{\text{rank}_{\alpha}(u)}\sigma_k^{\alpha} v_k^{\alpha}(x_{\alpha}) v_k^{\alpha^c}(x_{\alpha^c}).
\end{equation}
Here, $\sigma_1^{\alpha} \ge \hdots \ge \sigma_{\text{rank}_{\alpha}(u)}^{\alpha}$ are the $\alpha$-singular values, which are assumed to be sorted in decreasing order, and $v_k^{\alpha}\in L^2_{\mu_{\alpha}}$ and $v_k^{\alpha^c}\in L^2_{\mu_{\alpha^c}}$ are respectively the left and right normalized singular functions, such that $\Vert v_k^{\alpha} \Vert_{L^2_{\mu_{\alpha}}} = \Vert v_k^{\alpha^c} \Vert_{L^2_{\mu_{\alpha^c}}} = 1$.	For $r_\alpha\leq \text{rank}_{\alpha}(u)$, the truncated singular value decomposition of $u$ up to the rank $r_{\alpha}$ is then given by
\begin{equation}
u_{r_{\alpha}}(x) = \sum_{k=1}^{r_{\alpha}}\sigma_k^{\alpha} v_k^{\alpha}(x_{\alpha})  v_k^{\alpha^c}(x_{\alpha^c}).
\end{equation}

The $r_{\alpha}$ dominant left singular functions $\{v_{k}^{\alpha}\}_{k=1}^{r_\alpha}$ are called the $\alpha$-principal components of $u$, while the linear span of these $r_\alpha$ functions, denoted by $U_{\alpha}$, is called the $\alpha$-principal subspace of $u$. The function $u_{r_{\alpha}}=\mathcal{P}_{U_{\alpha}}u$ is the best approximation of $u$ with $\alpha$-rank $r_{\alpha}$, i.e. it satisfies
\begin{equation}
\label{eq:sol_principal_subs}
\Vert u - \mathcal{P}_{U_{\alpha}}u \Vert = \min_{\substack{v\in L_{\mu}^2 \\  \text{rank}_{\alpha}(v) \le r_{\alpha}}}  \Vert u - v\Vert = \min_{\substack{W_{\alpha} \subset L^2_{\mu_{\alpha}} \\ \dim (W_{\alpha}) = r_{\alpha} }}  \Vert u - \mathcal{P}_{W_{\alpha}}u \Vert := e_{r_{\alpha}}^{\alpha}(u).
\end{equation}

\paragraph{Approximation of the $\alpha$-principal subspaces.}
~
In practice, we do not directly determine the $\alpha$-principal subspaces of $u$, but an approximation of $U_{\alpha}$ is searched in a certain finite-dimensional subspace of $L_{\mu_{\alpha}}^2$, denoted $V_{\alpha}$. Noting $m_{\alpha} := \text{dim}(V_{\alpha}) \ge r_{\alpha}$, this approximation can be obtained by solving
\begin{equation}
\label{eq:sol_principal_subs_approx_space}
\min_{\substack{\dim (W_{\alpha}) = r_{\alpha} \\ W_{\alpha} \subset V_{\alpha}}}  \Vert u - \mathcal{P}_{W_{\alpha}}u \Vert^2 = \min_{\substack{\dim (W_{\alpha}) = r_{\alpha} \\ W_{\alpha} \subset V_{\alpha}}}  \Vert u - \mathcal{P}_{V_{\alpha}}u \Vert^2 + \Vert \mathcal{P}_{V_{\alpha}}u - \mathcal{P}_{W_{\alpha}}u \Vert^2.
\end{equation}
If $W_{\alpha} \subset V_{\alpha}$, $\mathcal{P}_{W_{\alpha}}  = \mathcal{P}_{W_{\alpha}}\mathcal{P}_{V_{\alpha}}$, and solving \Cref{eq:sol_principal_subs_approx_space} is equivalent to solving
\begin{equation*}
\min_{\substack{\dim (W_{\alpha}) = r_{\alpha} \\ W_{\alpha} \subset L^2_{\mu_{\alpha}}}}  \Vert \mathcal{P}_{V_{\alpha}}u - \mathcal{P}_{W_{\alpha}}\mathcal{P}_{V_{\alpha}}u \Vert_{L_{\mu}^2}^2,
\end{equation*}
whose solution is the $\alpha$-principal subspace of $\mathcal{P}_{V_{\alpha}}u$.\\
Since the orthogonal projection is usually not computable, $\mathcal{P}_{V_{\alpha}}$ is replaced by an oblique projection $\mathcal{Q}_{V_{\alpha}}$ from $L^2_{\mu_{\alpha}}$ onto $V_{\alpha} \otimes L^2_{\mu_{\alpha^c}}$. An approximate $\alpha$-principal subspace is then obtained by solving
\begin{equation}
\label{eq:sol_approx_principal_subs}
\min_{\substack{\dim (W_{\alpha}) = r_{\alpha} \\ W_{\alpha} \subset L^2_{\mu_{\alpha}}}}  \Vert \mathcal{Q}_{V_{\alpha}} u - \mathcal{P}_{W_{\alpha}}\mathcal{Q}_{V_{\alpha}}u \Vert_{L_{\mu}^2}^2,
\end{equation} 
whose solution $U_{\alpha}^{\star}$ is the $\alpha$-principal subspace of $\mathcal{Q}_{V_{\alpha}}u$.  
For each $V_{\alpha}$, $\mathcal{Q}_{V_{\alpha}}$ may be a sample-based projection. In the case where the samples used to define $\mathcal{Q}_{V_{\alpha}}$ are random, it is important to notice that the quantity $ \Vert \mathcal{Q}_{V_{\alpha}} u - \mathcal{P}_{U_{\alpha}^{\star}}\mathcal{Q}_{V_{\alpha}}u \Vert_{L_{\mu}^2}^2$ is thus a random variable.
\subsection{Choice of the oblique projection}\label{subs:choice_oblique_proj}
~
Here, we consider for $\mathcal{Q}_{V_{\alpha}}$ the boosted least-squares projection presented in \cite{HaNoPe2019}, whose main characteristics are now recalled.\\

Let $\{\varphi_j^{\alpha}\}_{j=1}^{m_{\alpha}}$ be an orthonormal basis of a $m_{\alpha}$-dimensional space $V_{\alpha} \subset L^2_{\mu_{\alpha}}$, and $\rho_\alpha$ be the measure defined by
\begin{equation}
\label{eq:optimal_measure}
d\rho_{\alpha}(x_{\alpha}) = w^{\alpha}(x_{\alpha})^{-1}d\mu_{\alpha}(x_{\alpha}), \ \  w^{\alpha}(x_{\alpha})^{-1} = \frac{1}{m_{\alpha}}\sum_{j=1}^{m_{\alpha}} \varphi_j^{\alpha}(x_{\alpha})^2.
\end{equation}
The function $w^{\alpha}(x_{\alpha})^{-1}$ is the density of $\rho_{\alpha}$ with respect to the reference measure $\mu_{\alpha}$. As it is invariant by rotation of $\{\varphi_j^{\alpha}\}_{j=1}^{m_{\alpha}}$, $\rho_{\alpha}$ does not depend on the chosen orthonormal basis but only on $V_\alpha$. For all $f^{\alpha} \in L_{\mu_{\alpha}}^2$, the boosted optimal weighted least-squares projection of $f^{\alpha}$ on $V_{\alpha}$, denoted $Q_{V_{\alpha}}$, is defined by
\begin{equation*}
Q_{V_{\alpha}}f^{\alpha} = \arg\min_{g^{\alpha} \in V_{\alpha}} \Vert f^{\alpha}  - g^{\alpha}\Vert_{\bm{x}_{\alpha}^{z_{\alpha}}},
\end{equation*}
with $\bm{x}_{\alpha}^{z_{\alpha}} := \{ x_{\alpha}^i \}_{i=1}^{z_{\alpha}}$ a set of $z_{\alpha}$ points in $\mathcal{X}_{\alpha}$ and $\Vert \cdot \Vert_{\bm{x}_{\alpha}^{z_{\alpha}}}$ the following discrete semi-norm 
\begin{equation*}
\Vert f^{\alpha} \Vert_{\bm{x}_{\alpha}^{z_{\alpha}}}^2 = \frac{1}{z_{\alpha}} \sum_{i=1}^{z_{\alpha}} w^{\alpha}(x_{\alpha}^{i})f^{\alpha}(x_{\alpha}^{i})^2.
\end{equation*}
An important aspect of the boosted least-squares projection is the fact that the chosen points $x_{\alpha}^{1},\ldots,x_{\alpha}^{z_{\alpha}}$ are realizations of dependent random variables whose measure is related to the measure $\rho_{\alpha}$ from Equation \Cref{eq:optimal_measure}. To select these $z_{\alpha}$ points in $\mathcal{X}_{\alpha}$, we draw $M$ times a $n_{\alpha}$-sample according to the product measure $\rho_{\alpha}^{\otimes n_{\alpha}}$ and select in this collection of $M$ samples the one minimizing a stability criterion (based on the empirical Gram matrix). We resample in this way, until a stability condition is verified. In a second time, we remove from this selected sample as many points as possible while maintaining the stability condition and guaranteeing a resulting number of samples $z_{\alpha}$ higher than $n_{\alpha,\min}=p_r n_\alpha$, with $p_r$ a constant independent of $m_\alpha$. For more details on the sampling procedure, see \cite{HaNoPe2019}. This sampling procedure allows us to ensure in expectation the stability of the projection. More precisely, \cite[Theorem 3.6]{HaNoPe2019} states that for any $f^{\alpha} \in L^2_{\mu_{\alpha}}$ and a fixed space $V_\alpha$, 
\begin{equation}
\label{eq:s-BLS_qo_projection}
\mathbb{E}(\Vert f^{\alpha} - Q_{V_{\alpha}}f^{\alpha} \Vert^2) \le \left(1 + \gamma \right) \Vert f^{\alpha} - P_{V_{\alpha}}f^{\alpha} \Vert^2,
\end{equation}
with $\gamma$ a constant that depends on $M$ and $p_r$, $\gamma = (1-\delta)^{-1}(1-\eta^M)^{-1}M $.  In the case where $V_{\alpha}$ is random, we can prove that
\begin{equation}
\label{eq:s-BLS_qo_projection-random}
\mathbb{E}(\Vert f^{\alpha} - Q_{V_{\alpha}}f^{\alpha} \Vert^2) \le \left(1 + \gamma \right) \mathbb{E}(\Vert f^{\alpha} - P_{V_{\alpha}}f^{\alpha} \Vert^2).
\end{equation}
By extension, the oblique projection $\mathcal{Q}_{V_{\alpha}}$ from $L^2_{\mu}$ to $V_{\alpha} \  \otimes \ L^2_{\mu_{\alpha^c}}$ such that $(\mathcal{Q}_{V_{\alpha}}u)(\cdot, x_{\alpha^c}) = Q_{V_{\alpha}}u(\cdot, x_{\alpha^c})$ is called a boosted weighted least-squares projection. Given some condition on the number of samples $n_{\alpha}$, the following lemma and theorem provide a stability result of the projection $Q_{V_{\alpha}}$ (their proofs are given in Appendix \ref{appendix:proof_lemma21} and \ref{appendix:proof_theorem22}).
\begin{lemma}\label{lem:bound_bls_projection}
	Let $Q_{V_{\alpha}}$ be the boosted least-squares projection verifying Equation \Cref{eq:s-BLS_qo_projection-random} for all $f^{\alpha} \in L^2_{\mu_{\alpha}}$.
	Let $\eta$ and $\delta$ be two constants, such that $0< \eta, \delta < 1$. If $n_{\alpha} \ge (- \delta +(1+\delta)\log(1+\delta))m_{\alpha}\log(2m_{\alpha}\eta^{-1})$, then for all $u\in L^2_{\mu}$, it holds
	\begin{equation*}
	\mathbb{E}(\Vert \mathcal{Q}_{V_{\alpha}}u \Vert^2) \le 2\left(1 + \gamma\right) \Vert u \Vert^2, \text{ with } \gamma \text{ defined by } (1-\delta)^{-1}(1-\eta^M)^{-1}M.
	\end{equation*}
\end{lemma}
We deduce the following quasi-optimality result when approximating the principal subspaces of $u$ by those of $\mathcal{Q}_{V_{\alpha}} u$.
\begin{theorem}\label{th:qo_constant_bls}
	Under the same hypotheses and notations as in Lemma \ref{lem:bound_bls_projection}, for all $u \in L^2_{\mu}$,
	\begin{equation}\label{eq:boosted_projec_inequality}
	\mathbb{E}(\Vert \mathcal{Q}_{V_{\alpha}} u - \mathcal{P}_{U^{\star}_{\alpha}}\mathcal{Q}_{V_{\alpha}}u\Vert^2) \le 2\left(1 + \gamma \right) e_{r_{\alpha}}^{\alpha}(u)^2,
	\end{equation}
	where $e_{r_{\alpha}}^{\alpha}(u) = \Vert u - \mathcal{P}_{U_{\alpha}}u\Vert$ and $\Vert \mathcal{Q}_{V_{\alpha}} u - \mathcal{P}_{U^{\star}_{\alpha}}\mathcal{Q}_{V_{\alpha}}u\Vert$ are respectively the minimal reconstruction errors of $u$ and $\mathcal{Q}_{V_{\alpha}}u$ associated to the $\alpha$-principal subspaces $U_{\alpha}$ and $ U^{\star}_{\alpha}$ defined in Equations \Cref{eq:sol_principal_subs} and \Cref{eq:sol_approx_principal_subs} respectively.
\end{theorem}

\subsection{Estimation of the $\alpha$-principal subspaces}\label{subs:estimation_pcs_approx}
\subsubsection{Accuracy of the empirical $\alpha$-principal subspaces}
~
Let $X_{\alpha^c}$ be a random vector associated with the measure $\mu_{\alpha^c}$. Hence, the approximation $U_{\alpha}^{\star}$ of the $\alpha$-principal subspace, which is solution of Equation \Cref{eq:sol_approx_principal_subs}, is equivalently defined as the solution of
\begin{equation*}
\min_{\text{dim}(U_{\alpha}^{\star}) = r_{\alpha}} \mathbb{E}\left(\Vert Q_{V_{\alpha}}u(\cdot, X_{\alpha^c}) - \mathcal{P}_{U_{\alpha}^{\star}} Q_{V_{\alpha}}u(\cdot, X_{\alpha^c})\Vert^2_{L^2_{\mu_{\alpha}}
}\right),
\end{equation*}
where $Q_{V_{\alpha}}u(\cdot, X_{\alpha^c})$ is now a function-valued random variable. An estimation of $U_{\alpha}^{\star}$, denoted $\widehat{U}_{\alpha}^{\star}$, can then be obtained using $z_{\alpha^c}$ independent and identically distributed (i.i.d) samples of $X_{\alpha^c}$, noted $\{x_{\alpha^c}^l\}_{l=1}^{z_{\alpha^c}}$, and by solving
\begin{equation}\label{eq:empirical_pcs}
\min_{\dim(\widehat{U}_{\alpha}^{\star})=r_{\alpha}}\frac{1}{z_{\alpha^c}} \sum_{l=1}^{z_{\alpha^c}} \Vert Q_{V_{\alpha}}u(\cdot, x_{\alpha^c}^l) - \mathcal{P}_{\widehat{U}_{\alpha}^{\star}}Q_{V_{\alpha}}u(\cdot, x_{\alpha^c}^l) \Vert_{L^2_{\mu_{\alpha}(\mathcal{X}_{\alpha})}}^2.
\end{equation}
See Appendix \ref{appendix:estimation_pcs_approx} for the practical solution of \Cref{eq:empirical_pcs}.

\begin{remark}
	The determination of $\hat{U}_{\alpha}^{\star}$ depends on the samples $\{x_{\alpha^c}^l\}_{l=1}^{z_{\alpha^c}}$ but also on the projection $\mathcal{Q}_{V_{\alpha}}$ and thus on the samples $\{x_{\alpha}^i\}_{i=1}^{z_{\alpha}}$.
\end{remark}
\begin{remark} An interesting question is to compare the behavior of the reconstruction error $\Vert \mathcal{Q}_{V_{\alpha}}u - \mathcal{P}_{\hat{U}_{\alpha}^{\star}}\mathcal{Q}_{V_{\alpha}}u \Vert $ associated with the empirical subspace $\hat{U}_{\alpha}^{\star}$ with the minimal reconstruction error $\Vert \mathcal{Q}_{V_{\alpha}}u - \mathcal{P}_{U_{\alpha}^{\star}}\mathcal{Q}_{V_{\alpha}}u \Vert $ associated with $U_{\alpha}^{\star}$. In \cite{Wahl2020}, the authors derive high-probability bounds for the reconstruction error of empirical principal subspaces under strong assumptions on the function $u$, that are hardly verified in practice. Also, in \cite{CoNoGi2020}, the authors show that in the case where the minimal reconstruction error $e_{r_{\alpha}}^{\alpha}(u)$ has a certain algebraic decay, the same rate of convergence can be obtained for the reconstruction error of the empirical subspaces if the number of samples $z_{\alpha^c}$ is chosen sufficiently high, but this condition seems to be pessimistic in many practical cases. A major difficulty to obtain a similar result in our setting comes from the fact that we consider the $\alpha$-principal subspaces of $\mathcal{Q}_{V_{\alpha}}u$, not of $u$. Choosing a sample-based projection $\mathcal{Q}_{V_{\alpha}}$ where the samples are not deterministic but randomly drawn from a certain measure implies that $\mathcal{Q}_{V_{\alpha}}$ is random and depends on samples of the function $u$, which makes tricky the interpretability of the hypotheses made on $u$. For these reasons, in the next section, we propose an adaptive strategy to estimate the empirical $\alpha$-principal subspaces $\widehat{U}_{\alpha}^{\star}$ with a given tolerance in order to choose a small number of samples $z_{\alpha^c}$.
\end{remark}
\subsubsection{Adaptive estimation of the $\alpha$-principal subspaces}\label{subs:adaptive_pca}
~
For a given number of samples $z_{\alpha^c}$, the reconstruction error of the empirical $\alpha$-principal subspace $\widehat{U}_{\alpha}^{\star}$ is estimated by leave-one-out cross validation. While this error is greater than the desired tolerance $\varepsilon$, we increase the dimension of $\widehat{U}_{\alpha}^{\star}$. If for $\dim(\widehat{U}_{\alpha}^{\star}) =z_{\alpha^c}$, the tolerance is not reached, we increase the number of samples $z_{\alpha^c}$ and again estimate the leave-one-out error. We start from $z_{\alpha^c}=1$ and impose an upper bound, $z_{\alpha^c} \le k_{PCA} m_{\alpha}$, where $k_{PCA} \in \mathbb{N}^{\star}$ is a sampling factor. This procedure, presented in Appendix \ref{appendix:estimation_pcs_approx} in Algorithm \ref{algo:adapt_pcs_precision}, provides in many experiments a small number of samples $z_{\alpha^c}$ to get the desired accuracy.

\section{Learning tree tensor networks using PCA}\label{sec:learning_ttnetworks_PCA}
~
In this section, we present an algorithm that constructs an approximation of a function $u \in L^2_{\mu}$ in tree-based tensor format. After briefly recalling the definition of tree tensor networks, we present in detail the algorithm we propose, and then show to what extent it is possible to bound the error of the resulting approximation.

\subsection{Dimension partition tree}
\begin{figure}[h!]
	\begin{center}
\includegraphics{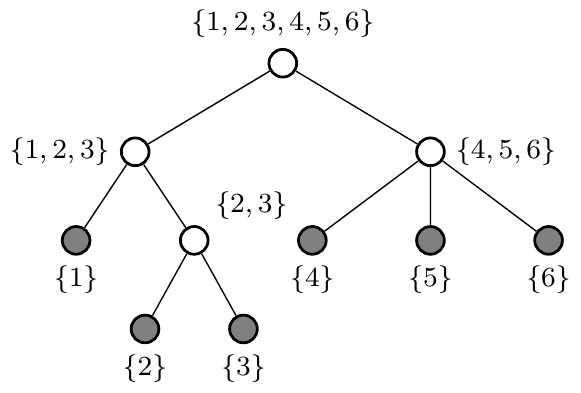}		
		\caption{Dimension partition tree over $\{1, \hdots, 6\}$ with its leaves represented in gray}\label{fig:treewithleaves}
		\end{center}
\end{figure}
A dimension partition tree $T$ over $D = \{1, \hdots, d\}$ is a collection of subsets in $D$ having the following properties: 
\begin{itemize}
	\item $D$ is the root of the tree $T$, \item a node $\alpha \in T$ is a non empty subset of $D$, whose cardinality is denoted by $\# \alpha$, 
	\item for each node $\alpha \in T$, the set of sons $S(\alpha)$ of $\alpha$ is either empty (for $\# \alpha =1$) or forms a partition of $\alpha$ with $\#S(\alpha) \ge 2$.
\end{itemize}
The nodes $\alpha$ such that $S(\alpha) = \emptyset$ are the leaves of the tree $T$ and the set containing all leaves is denoted $\mathcal{L}(T)$. As an illustration, Fig. \ref{fig:treewithleaves} shows a particular dimension partition tree, with $d =6$, and
\begin{equation}
T = \{\{1\},\{2\},\{3\},\{4\}, \{5\}, \{6\}, \{2,3\},\{1,2,3\},\{4,5,6\},\{1,2,3,4,5,6\}\}.
\end{equation} 
For a node $\alpha$, $l(\alpha)$ denotes the level of the node $\alpha$ in the tree $T$. It is defined recursively from the root to the leaves, such that $l(D) = 0$ and if $\beta \in S(\alpha)$, $l(\beta) =l(\alpha) + 1$. The maximum level of the nodes in $T$ is the depth of the tree: $\text{depth}(T) = \max_{\alpha \in T} l(\alpha)$.
\subsection{Tree tensor networks}
For $T$ a dimension tree over $D$, we define the $T$-rank of a function $v$, $\text{rank}_{T}(v)$, as the tuple $\text{rank}_{T}(v) = \{\text{rank}_{\alpha}(v)\}_{\alpha \in T}$. Then, we define an approximation format, $\mathcal{T}_r^T(V)$ which is the set of functions in some subspace $V \subset L^2_{\mu}$ with $T$-rank bounded by $r = (r_{\alpha})_{\alpha \in T}$,
\begin{equation}\label{def:Tranktensors}
\mathcal{T}_r^T(V) = \{v \in V : \text{rank}_T(v) \le r\} = \bigcap_{\alpha \in T} \{v \in V : \text{rank}_{\alpha}(v) \le r_{\alpha} \}.
\end{equation}
Elements of $\mathcal{T}_r^T(V)$ are tree tensor networks. A function $v \in \mathcal{T}_r^T(V)$ can be written under the form (\ref{eq:svd_multi_function})
\begin{equation*}
v(x) = \sum_{k=1}^{\text{rank}_{\alpha}(v)}\sigma_k^{\alpha} v_k^{\alpha}(x_{\alpha})  v_k^{\alpha^c}(x_{\alpha^c}), \text{ for each } \alpha \in T
\end{equation*}
\subsection{Description of the algorithm}
~
For a given dimension tree $T$, the algorithm we propose to determine the parameters of a tree tensor network approximation of $u^{\star}$ relies on a leaf-to-root exploration of $T$, a sequential estimation of $\alpha$-principal subspaces (see Appendix \ref{appendix:estimation_pcs_approx} for more details), and a final least-squares projection of $u$ on a product of subspaces.\\ 

The first step of the algorithm consists in computing estimations $\widehat{U}_{\alpha}^{\star}$ of $\alpha$-principal subspace of $u$ for each node of the tree $\alpha \in T \setminus \{D\}$. As explained in Section \ref{sec:alpha_principal_subspaces}, each subspace $\widehat{U}_{\alpha}^{\star}$ is searched in a finite-dimensional subspace $V_\alpha$. Depending on the position of the node $\alpha$, two cases are distinguished.
On the one hand, for a leaf node $\alpha \in \mathcal{L}(T)= \{  \{1\},\ldots,\{ d \} \}$, $V_{\alpha}$ is a given finite dimensional space in $L_{\mu_{\alpha}}^2(\mathcal{X}_{\alpha})$ (e.g. splines, wavelets, polynomials, ...).
On the other hand, for an internal node $\alpha \notin \mathcal{L}(T)$, $V_{\alpha}$ is chosen equal to $\otimes_{\beta \in S(\alpha)}\widehat{U}_{\beta}^{\star}$, that is to the tensor product space of the approximated $\alpha$-principal subspaces of the sons of $\alpha$. A each subspace $\widehat{U}_{\beta}^{\star}$ is a statistical estimation based on evaluations of $u$ at randomly chosen points in $\mathcal{X}$, $V_{\alpha}$ is a random space.

The second step of the algorithm is the projection of the function $u$ on the tensor product space formed by the $\alpha$-principal subspaces of the sons of the root of the tree, $S(D)$, that is to say
\begin{equation}
u^{\star} = \mathcal{Q}_{V_{D}}u \text{ where } V_D = \bigotimes_{\alpha \in S(D)} \widehat{U}_{\alpha}^{\star}, 
\end{equation}
with $\mathcal{Q}_{V_{D}}$ a boosted optimal least-squares projection. The final approximation is in $\mathcal{T}_r^T(V)$ with $r = (r_{\alpha})_{\alpha \in T}$ and $r_{\alpha} = \dim(\hat{U}_{\alpha}), \alpha \in T \setminus \{D\}$ and $V = \bigotimes_{\nu = 1}^d V_{\nu}$. A synthetic description of this procedure is summarized in Algorithm \ref{algo:constr_approx_TBT}.\\

\begin{algorithm}[h!]
	\caption{Construction of a tree tensor network approximation}\label{algo:constr_approx_TBT}
	\begin{algorithmic}
		\STATE \hspace{-0.4cm}\textbf{Inputs:} dimension tree $T$, function to approximate $u$, measure $\mu$, finite-dimensional spaces $V_{\nu} \text{ for } \nu \in \mathcal{L}(T)$, desired tolerance $\varepsilon$.
		\STATE \hspace{-0.4cm}\textbf{Outputs:} approximation $u^{\star}$ in $\mathcal{T}_r^T(V)$
		\FOR{$\alpha \in T$ going by decreasing level}
		\IF{$\alpha \notin \mathcal{L}(T)$}
		\STATE Set $V_{\alpha} = \bigotimes_{\beta \in S(\alpha)}\widehat{U}^{\star}_{\beta}$
		\ENDIF
		\STATE Compute $\widehat{U}_{\alpha}^{\star} \subset V_{\alpha}$, the estimation of the $\alpha$-principal subspaces with relative reconstruction error $\varepsilon$, thanks to Algorithm \ref{algo:adapt_pcs_precision} (see Appendix \ref{appendix:estimation_pcs_approx} for a detailed description of this algorithm).
		\ENDFOR
		\STATE Set $V_D = \bigotimes_{\alpha \in S(D)} \widehat{U}_{\alpha}^{\star}$.
		\STATE Compute $u^{\star} = \mathcal{Q}_{V_D}u$.
	\end{algorithmic}
\end{algorithm}

\begin{remark}
	In practice, the boosted optimal weighted least-squares projection requires sampling from the optimal measure from Equation (\ref{eq:optimal_measure}), whose expression depends on the position of the node $\alpha$ in the tree (see Appendix \ref{app:explicitExpression} for explicit expressions of these optimal measures). 
\end{remark}

\subsection{Error analysis}\label{subs:error_analysis}
~
The following lemma provides a first error bound for the error of approximation without any assumption on the reconstruction error of the empirical $\alpha$-principal subspace $\hat{U}_{\alpha}^{\star}$.
\begin{lemma}\label{lem:bound_final_error_qo_proj}
	Assume that for all $\alpha \in T$, $Q_{V_{\alpha}}$ is the boosted optimal weighted least-squares projection verifying the assumptions from Theorem \ref{th:qo_constant_bls}.\\
	The error of approximation is bounded in expectation as follows,
	\begin{equation*}
	\mathbb{E}(\Vert u - u^{\star} \Vert^2) \le \sum_{\alpha \in T \setminus D} (2C_1)^{l(\alpha)}\mathbb{E}(\Vert \mathcal{Q}_{V_{\alpha}}u - \mathcal{P}_{\hat{U}_{\alpha}^{\star}}\mathcal{Q}_{V_{\alpha}}u \Vert^2) + \sum_{\alpha \in \mathcal{L}(T)} \frac{1}{2} (2C_1)^{l(\alpha)+1} e^{\alpha,dis}_{m_{\alpha}}(u)^2,
	\end{equation*}
	where $C_1:=2(\gamma+1)$, with $\gamma$ defined in Theorem \ref{th:qo_constant_bls} depending on the boosted optimal weighted least-squares projection $Q_{V_{\alpha}}$, $\Vert \mathcal{Q}_{V_{\alpha}}u - \mathcal{P}_{\hat{U}_{\alpha}^{\star}}\mathcal{Q}_{V_{\alpha}}u \Vert^2$ is the reconstruction error associated to $\hat{U}_{\alpha}^{\star}$, and $e^{\alpha,dis}_{m_{\alpha}}(u) = \Vert u - \mathcal{P}_{V_{\alpha}}u \Vert_{L_{\mu}^2}$ is the error of discretization due to the use of a finite-dimensional space $V_{\alpha}$ for the leaf $\alpha$.
\end{lemma}
Making further assumptions on the reconstruction error of the empirical $\alpha$-principal subspace $\hat{U}_{\alpha}^{\star}$, we deduce the theorem hereafter.
\begin{theorem}\label{th:final_err_bound}
	Assume that for all $\alpha \in T$, $Q_{V_{\alpha}}$ is the boosted optimal weighted least-squares projection verifying the assumptions from Theorem \ref{th:qo_constant_bls}.\\
	Assume that for all $\alpha \in T \setminus D$, the empirical $\alpha$-principal subspaces of $\mathcal{Q}_{V_{\alpha}}u$ solutions of Equation \Cref{eq:empirical_pcs}, denoted $\widehat{U}_{\alpha}^{\star}$, are such that the reconstruction errors verify
	\begin{equation}\label{eq:hyp_ineq_approx_empirical_subspaces}
	\mathbb{E}(\Vert \mathcal{Q}_{V_{\alpha}}u - \mathcal{P}_{\widehat{U}_{\alpha}^{\star}}\mathcal{Q}_{V_{\alpha}}u \Vert^2\vert \mathcal{Q}_{V_{\alpha}}u)  \le C_2 \mathbb{E}(\Vert \mathcal{Q}_{V_{\alpha}}u - \mathcal{P}_{U_{\alpha}^{\star}}\mathcal{Q}_{V_{\alpha}}u \Vert^2\vert \mathcal{Q}_{V_{\alpha}}u ),
	\end{equation}
	where $\Vert \mathcal{Q}_{V_{\alpha}}u - \mathcal{P}_{U_{\alpha}^{\star}}\mathcal{Q}_{V_{\alpha}}u \Vert^2$ is the reconstruction error associated with the $\alpha$-principal subspace of $U_{\alpha}^{\star}$ solution of Equation \Cref{eq:sol_approx_principal_subs}. Then the error of approximation is bounded in expectation as follows: 
	\begin{equation}\label{eq:error_bound}
	\mathbb{E}(\Vert u - u^{\star} \Vert^2) \le C_1C_2 \sum_{\alpha \in T \setminus D} (2C_1)^{l(\alpha)} e_{r_{\alpha}}^{\alpha}(u)^2 + \sum_{\alpha \in \mathcal{L}(T)} \frac{1}{2} (2C_1)^{l(\alpha)+1} e_{m_{\alpha}}^{\alpha,dis}(u)^2,
	\end{equation}
	where $C_1$ and $e_{m_{\alpha}}^{\alpha,dis}(u)$ are defined in Lemma \ref{lem:bound_final_error_qo_proj}.
\end{theorem}

In the upper bound from \Cref{eq:error_bound}, 
the first term is related to the error in the estimation of the principal components, while the second term comes from the discretization error due to the introduction of feature spaces. Assumption \Cref{eq:hyp_ineq_approx_empirical_subspaces}
is related to the discussion from section \ref{subs:estimation_pcs_approx}. From \Cref{eq:error_bound}, noting that 
for all $\alpha$, $e^\alpha_{r_\alpha}(u)$ and $e_{m_{\alpha}}^{\alpha,dis}(u)$ are bounded by the best approximation error in $ \mathcal{T}_r^T(V)$, 
we deduce a quasi-optimality result in expectation
$$
\mathbb{E}(\Vert u - u^{\star} \Vert^2) \le \widetilde C \min_{v\in \mathcal{T}_r^T(V)} \Vert u - v \Vert^2 ,
$$ 
with constant $\widetilde C$ depending on $C_1$, $C_2$ and the dimension tree.

%

\section{Tree adaptation}\label{sec:tree_adaptation}	
~
The choice of the tree may have a significant impact on the complexity required to reach a certain precision. Several numerical illustrations that underline this issue are presented in \cite{GrelierNouyChevreuil2018} or \cite{Haberstich2020}. 

In this section, we propose a strategy to find a tree $T$ with the objective of reducing the number of evaluations necessary to get a certain accuracy.
Other strategies that aims at performing tree optimization are presented in \cite{Haberstich2020} and compared on numerical examples.
The proposed strategy includes the tree optimization inside the algorithm for the construction of the approximation (Algorithm \ref{algo:constr_approx_TBT} presented in Section \ref{sec:learning_ttnetworks_PCA}). 
As it will be explained in Section \ref{par:complexity_analysis}, the number of evaluations $n$ necessary to get a desired accuracy $\varepsilon$ using this algorithm is related to the storage complexity, defined by
\begin{equation}\label{eq:storage_complexity_global}
\mathcal{S}(r,T)=\sum_{\alpha \notin \mathcal{L}(T)} r_{\alpha} \prod_{\beta \in S(\alpha)}r_{\beta} + \sum_{\alpha \in \mathcal{L}(T)} r_{\alpha} m_{\alpha}, 
\end{equation}
with $r=r(\varepsilon)$ the ranks for achieving the precision $\varepsilon$.
For a given precision $\varepsilon$, as the storage complexity associated to the leaves nodes $\sum_{\alpha \in \mathcal{L}(T)} r_{\alpha}(\varepsilon) m_{\alpha}$ is independent of the choice of the tree, minimizing $\mathcal{S}(r(\varepsilon),T)$ amounts at minimizing the following cost function
\begin{equation}\label{eq:cost_function_global}
\mathcal{C}(r(\varepsilon),T)=\sum_{\alpha \notin \mathcal{L}(T)} r_{\alpha}(\varepsilon) \prod_{\beta \in S(\alpha)}r_{\beta}(\varepsilon).
\end{equation}

Hence, to minimize $\mathcal{C}(r(\varepsilon),T)$, it seems interesting to look for a strategy which reduces the $\alpha$-ranks $r_{\alpha}$, for all interior nodes $\alpha \in T \setminus \mathcal{L}(T)$. \\

At this point, we can list two major difficulties for this objective. First, as the number of possible trees $T$ scales exponentially in the dimension $d$, finding the best tree is a combinatorial problem. Secondly, for each considered tree $T$, the $\alpha$-ranks $(r_\alpha(\varepsilon))_{\alpha \in T \setminus \mathcal{L}(T)}$ to reach a precision $\varepsilon$ are a priori unknown, and need to be estimated using evaluations of $u$. It is then obvious that an exhaustive search for the best dimension tree is completely unrealistic from a computational point of view in the context of costly evaluations. \\ 

To circumvent some of these difficulties, we propose to progressively construct a dimension partition tree by suitable pairings of variables.
By pairing variables from the leaves to the root, we indeed reduce sharply the number of possible trees (and thus the number of $\alpha$-ranks to be evaluated). However the number of remaining {pairings of variables} to explore may remain large, this is why we also propose a stochastic strategy, that will select randomly a reduced number of {pairings} but preferentially the ones with low $\alpha$-ranks.

\subsection{Estimation of $\alpha$-ranks}\label{subs:eps_rank_estimation}
~
Performing tree optimization requires the estimations of $\alpha$-ranks $r_{\alpha}(\varepsilon)$ for reaching a precision $\varepsilon$. These estimations require evaluations of the function $u$, and this cost (denoted $n_{optim}$) should be reasonable compared to the number $n$ of evaluations required for constructing the approximation for a given tree. \\

To estimate these $\alpha$-ranks, a strategy based on an adaptive cross approximation technique \cite{Bebendorf2000} is proposed in \cite{BallaniGrasedyck2014}. Inspired by this work, we propose in the following a strategy based on leave-one-out cross validation to estimate the $\alpha$-ranks $r_{\alpha}(\varepsilon)$ to achieve an empirical relative error $\varepsilon$. To do this, we consider the matrix of the evaluations of $u$, $\bm{B}^{\alpha} = \{ u(x_{\alpha}^l, x_{\alpha^c}^k)  : 1 \le l \le z_{\alpha}, 1 \le k \le z_{\alpha^c}\}$, where $\{x_{\alpha}^l\}_{l=1}^{z_{\alpha}}$ are i.i.d samples of $X_{\alpha}$ and $\{x_{\alpha^c}^k\}_{k=1}^{z_{\alpha^c}}$ are i.i.d samples of $X_{\alpha^c}$. We introduce 
$\bm{B}^{\alpha}_{\setminus i}$, the matrix $ \bm{B}^{\alpha}$ without the column $i$, which admits a singular value decomposition 
\begin{equation}
\bm{B}^{\alpha}_{\setminus i} = \sum_{k \ge 1}\sigma^{\setminus i,k}_{\alpha}\bm{v}^{\setminus i,k}_{\alpha}(\bm{v}^{\setminus i,k}_{\alpha^c})^T,
\end{equation}
where $\sigma^{\setminus i,k}_{\alpha}$ are the singular values sorted in decreasing order, $\bm{v}^{\setminus i,k}_{\alpha}$ and $\bm{v}^{\setminus i,k}_{\alpha^c}$ are respectively the left and right singular vectors of $\bm{B}^{\alpha}_{\setminus i}$.
For all $r  \in \{ 1, \hdots, \min(z_{\alpha}, z_{\alpha^c}) \}$, let $\bm{V}^{\alpha}_{\setminus i,r}$ be the matrix whose columns are $(\bm{v}_{\setminus i,1}^{\alpha}, \hdots, \bm{v}_{\setminus i,r}^{\alpha}) $. The rank $r_{\alpha}(\varepsilon)$ is then estimated as the minimal integer such that
\begin{equation}
\frac{1}{z_{\alpha^c}} \sum_{i=1}^{z_{\alpha^c}} \Vert \bm{B}^{\alpha}_i - \bm{V}^{\alpha}_{\setminus i,r_{\alpha}(\varepsilon)}(\bm{V}^{\alpha}_{\setminus i,r_{\alpha}(\varepsilon)})^T \bm{B}^{\alpha}_i \Vert_2^2  \le \varepsilon^2 \frac{1}{z_{\alpha^c}} \sum_{i=1}^{z_{\alpha^c}} \Vert \bm{B}^{\alpha}_i \Vert_2^2,
\end{equation}
where $\bm{B}^{\alpha}_i$ denotes the $i^{th}$ column of $\bm{B}^{\alpha}$.

\begin{remark}
	Estimating the $\alpha$-ranks yielding a small precision $\varepsilon$ may require many evaluations. It is important to underline that to perform tree optimization we do not need to know the exact value of $r_{\alpha}(\varepsilon)$ but we want to have an estimation enough accurate to detect whether $r_{\alpha}(\varepsilon)$ is high or not. Therefore, $r_{\alpha}$ is estimated with a coarse precision. For the sake of conciseness, the whole strategy is described in Appendix \ref{appendix:alpha_rank_estimation} with the Algorithm \ref{algo:eval_eps_ranks}.\\
\end{remark}

\subsection{Leaves-to-root construction of the tree with stochastic optimizations}\label{subs:local_optimization}
~

In this section, we present the new strategy that progressively constructs a dimension partition tree by suitable pairings of variables, where these pairings are stochastically explored.\\

Let $\Lambda = \{\alpha_1 \hdots, \alpha_l \}$ be a partition of $D = \{1, \hdots, d\}$. When $l= \# \Lambda$ is even, we consider $\mathcal{J}(\Lambda)$ the set of all partitions of $\Lambda$ where each element has a cardinal equal to two. Each partition $\Gamma \in \mathcal{J}(\Lambda)$ thus contains $\frac{l}{2}$ elements. When $\# \Lambda$ is odd, we consider the set $\mathcal{J}(\Lambda) = \bigcup_{\alpha \in \Lambda} \bigcup_{\Lambda \in \mathcal{J}(\Gamma \setminus \alpha) } \{ \{ \alpha \} \ \cup  \Gamma  \}$. Among all partitions of $\mathcal{J}(\Lambda)$, the aim is to find the one, noted $\Gamma$, which minimizes
\begin{equation}\label{eq:cost_function_local}
\mathcal{C}_l(\Gamma) =  \sum_{\beta \in \Gamma} r_{\beta}(\varepsilon) \prod_{\substack{\alpha \in \Lambda\\ \alpha \subset \beta}} r_{\alpha}(\varepsilon).
\end{equation}

In practice, computing the function $\mathcal{C}_l(\Gamma)$ for all $\Gamma \in \mathcal{J}(\Lambda)$ requires a lot of $\alpha$-ranks estimations and it is therefore not affordable. To minimize $\mathcal{C}_l(\Gamma)$, we propose a stochastic algorithm which finds a partition $\Gamma$ associated to a minimal cost function $\mathcal{C}_l(\Gamma)$ among a limited set of partitions. The principle is to compare a current partition $\Gamma$ of $\mathcal{J}(\Lambda)$ with a new one $\Gamma^{\star}$ obtained from $\Gamma$ by permuting two nodes selected according to a probability distribution defined hereafter, and to accept $\Gamma^{\star}$ if $\mathcal{C}_l(\Gamma^{\star}) < \mathcal{C}_l(\Gamma)$.\\

To select a potentially interesting permutation, we propose to choose the first node $\nu_1$ in $\Lambda$ according to the distribution 
\begin{equation}\label{eq:local_proba_first_node}
\mathbb{P}(\nu_1 = \alpha)\  \propto \  r_{P_{\Gamma}(\alpha)}(\varepsilon)^{\gamma_1}, \  \text{ where $P_{\Gamma}(\alpha)$ is the parent of $\alpha$ in $\Gamma$ }.
\end{equation}
A higher $\gamma_1$ increases the probability to select a node $\nu_1$ whose parent in $\Gamma$ has a high rank. Once the node $\nu_1$ is selected, we consider the set $\Lambda \setminus  (\{\nu_1\} \cup \{\nu_1^b\})$, where  $\nu_1^b$ is the second element of the pair formed with $\nu_1$ (in the case $\nu_1$ is a singleton $\nu_1^b = \emptyset$), that is to say $ P_{\Gamma}(\nu_1) = \nu_1 \cup \nu_1^b$. Then, we propose to draw the second node $\nu_2$ in $\Lambda \setminus  (\{\nu_1\} \cup \{\nu_1^b\})$ according to the distribution
\begin{equation}\label{eq:local_proba_second_node}
\mathbb{P}(\nu_2 = \alpha \vert \nu_1)\  \propto \  r_{P_{\Gamma (\alpha)}}(\varepsilon)^{\gamma_2}, \  \text{ where }\alpha \in \Lambda \setminus (\{\nu_1\} \cup \{\nu_1^b\}).
\end{equation}
Again, a higher $\gamma_2$ increases the probability to select a node $\nu_2$ whose parent in $\Gamma$ has a high rank. If the permutation of the two nodes $\nu_1$ and $\nu_2$ decreases the cost function, then the two nodes are permuted. $n_P$ successive random permutations of the nodes are performed according to this distribution. The last partition $\Gamma$ is the one associated to the lowest cost function $\mathcal{C}_l(\Gamma)$ among all the visited partitions. A synthetic description of this strategy for selecting a dimension tree adapted to $u$ can be found in Algorithm \ref{algo:tree_local_optim}. \\	
\begin{algorithm}[h!]
	\caption{Optimization of nodes pairing}\label{algo:tree_local_optim}
	\begin{algorithmic}
		\STATE \hspace{-0.4cm} \textbf{Inputs:} function to approximate $u$, partition $\Lambda$ of $D$, maximal number of iterations $n_P$, $\gamma_1$, $\gamma_2$.
		\STATE \hspace{-0.4cm} \textbf{Outputs:} $\Gamma$
		\STATE Choose randomly $\Gamma \in \mathcal{J}(\Lambda)$.
		\STATE Calculate $\mathcal{C}_l(\Gamma)$ according to Eq. \Cref{eq:cost_function_local}, with estimation of the $\alpha$-ranks using Algorithm \ref{algo:eval_eps_ranks}.
		\FOR{$k=1, \hdots, n_P$}
		\STATE $\Gamma^{\star} \leftarrow \Gamma$
		\STATE Draw $\nu_1$ according to the distribution (\ref{eq:local_proba_first_node}) and then $\nu_2$ according to the distribution (\ref{eq:local_proba_second_node}).
		\STATE Calculate $\mathcal{C}_l(\Gamma^{\star})$ according to \Cref{eq:cost_function_local}, with estimation of the $\alpha$-ranks using Algorithm \ref{algo:eval_eps_ranks}.
		\IF{ $\mathcal{C}_l(\Gamma^{\star}) \le \mathcal{C}_l(\Gamma)$}
		\STATE $\Gamma \leftarrow \Gamma^{\star}$
		\ENDIF
		\ENDFOR
	\end{algorithmic}
\end{algorithm}

Then, the overall strategy that constructs the tree during the tree tensor network approximation of $u$ is given in Algorithm \ref{algo:adpat_constr_tree_local_optim}.\\

\begin{algorithm}[h!]	\caption{Adaptive construction of the tree with local optimization}\label{algo:adpat_constr_tree_local_optim}
	\begin{algorithmic}
		\STATE \hspace{-0.4cm} \textbf{Inputs:} function to approximate $u$, measure $\mu$, approximation spaces $V_{\alpha}, \alpha \in \mathcal{L}(T)$, tolerance $\varepsilon$, parameters relative to the rank estimation $\varepsilon_c$, $n_{\alpha}$, $n_{\alpha^c}$. 
		\STATE \hspace{-0.4cm} \textbf{Outputs:} the dimension tree $T$ and the approximation $u^{\star}$
		\STATE Set $\Lambda = \{\{1\}, \hdots, \{d\}\}$ and $T = \Lambda$
		\WHILE{$\# \Lambda > 1$} 
		\FOR{$\alpha \in \Lambda$}
		\IF{$\alpha \notin \mathcal{L}(T)$}
		\STATE Set $V_{\alpha} = \bigotimes_{\beta \in S(\alpha)}\widehat{U}^{\star}_{\beta}$
		\ENDIF
		\STATE Compute the estimation $\widehat{U}_{\alpha}^{\star}$ of the $\alpha$-principal subspace of $\mathcal{Q}_{V_{\alpha}}u$ with relative reconstruction error $\varepsilon$, using Algorithm \ref{algo:adapt_pcs_precision}.
		\ENDFOR
		\STATE Determine a partition $\Gamma$ of $D$ by pairing elements of $\Lambda$ thanks to Algorithm  \ref{algo:tree_local_optim}.
		\STATE Set $T \leftarrow T \cup \Gamma$ and $\Lambda \leftarrow \Gamma$
		\ENDWHILE
		\STATE Set $V_D =  \bigotimes_{\alpha \in S(D)} \widehat{U}_{\alpha}^{\star}$
		\STATE Compute $u^{\star} = \mathcal{Q}_{V_D}u$
		\STATE Set $T = T \cup D$.
	\end{algorithmic}
\end{algorithm}

\begin{remark}
	The efficiency of our strategy will be compared on numerical examples to the one from \cite{BallaniGrasedyck2014}, where the tree $T$ is adaptively constructed (with local deterministic optimization) during the algorithm \ref{algo:constr_approx_TBT} from Section \ref{sec:learning_ttnetworks_PCA}.\\
	
	The strategy proposed in \cite{BallaniGrasedyck2014} constructs a tree in a leaves-to-root strategy  by successive clusterings of disjoint subsets of $D = \{1, \hdots d\}$. $p$ is the number of elements gathered at the same time (which corresponds to the tree's arity) and it can be chosen to limit the number of possibilities which are explored. The clustering criterion is based on an estimation of the $\alpha$-ranks. The authors explain that when $p>3$ the computational cost for the adaptive part is much more higher than the cost necessary to compute the approximation in the tree tensor network. {As we only consider pairings, in our strategy, we set $p=2$ in all the numerical examples.}
\end{remark}

\section{Numerical examples}\label{sec:numerical_examples}
~
This section aims at showing the efficiency of the following three contributions:
\begin{itemize}
	\item Replacing a non-controlled projection (for example empirical interpolation as in \cite{Maday2007}) by the boosted least-squares projection from \cite{HaNoPe2019}, for which we can provide an explicit bound for the approximation error in expectation. We choose the same parameters for this projection in all the numerical examples: $M=100$, $\delta = 0.9$ and $\eta = 0.01$. The maximal proportion of samples to be removed $p_r$ is chosen equal to $\frac{m_{\alpha}}{n_{\alpha}}$, implying that points are removed while the stability condition is verified. 
	\item Using the adaptive strategy for the determination of the spaces $V_{\alpha}$ in the leaves (described in Algorithm \ref{algo:adapt_basis}), thanks to the adaptive boosted least-squares strategy from \cite{HaNoPe2021}.
	\item Using the adaptive strategy for the estimation of the $\alpha$-principal components presented in Algorithm \ref{algo:adapt_pcs_precision}. In this whole numerical part, the sampling factor $k_{PCA}$ is always taken equal to $3$, which is an arbitrary choice. When the principal components are not adaptively chosen, we simply take $z_{\alpha^c} = m_{\alpha}$ (using notations from Algorithm \ref{algo:adapt_pcs_precision}). 
\end{itemize}

To illustrate the efficiency of the strategies, we assess the quality of the approximation $u^{\star}$ of a function $u \in L^2_{\mu}$ by estimating the error of approximation by
$$ \varepsilon(u^{\star}) = \left(\frac{1}{n_{test}} \sum_{x \in \bm{x}_{test}} (u(x) - u^{\star}(x))^2\right)^{1/2},$$

\noindent{}where the elements of $\bm{x}_{test}$ are $n_{test}$ i.i.d. realizations of $X\sim\mu$.
In practice, we choose $n_{test} =1000$. To study the robustness of the methods, we compute $10$ times the approximations, draw $10$ different test samples $\bm{x}_{test}$ and compute empirical confidence intervals of level $10\%$ and $90\%$ for the errors of approximation.

\subsection{Complexity analysis}\label{par:complexity_analysis}

The total number of evaluations necessary to build the approximation in tree tensor network using Algorithm \ref{algo:constr_approx_TBT} depends on $z_{\alpha}$ and $z_{\alpha^c}$. $z_{\alpha}$ is the number of samples used to build the projection, and $z_{\alpha^c}$ is the number of samples used to estimate the $\alpha$-principal subspaces. For each node $\alpha \in T \setminus \{D\}$, the number of samples $N_{\alpha}$ needed to estimate the $\alpha-$principal subspace is in $\mathcal{O}(z_{\alpha}z_{\alpha^c})$. In addition, if the stability conditions of Theorem \ref{th:final_err_bound} are verified, $z_{\alpha}^c$ scales in $\mathcal{O}(m_{\alpha}\log(m_{\alpha}))$. Assuming that $z_{\alpha} $ scales in $\mathcal{O}(r_{\alpha})$, it comes
\begin{equation*}
\begin{aligned}
n  & = \sum_{\alpha \in T} N_{\alpha} = \mathcal{O}\left( \sum_{\alpha \in T}  m_{\alpha}\log(m_{\alpha})r_{\alpha}\right)\\
& = \mathcal{O}(\sum_{\alpha \in \mathcal{L}(T)} m_{\alpha}r_{\alpha}  + \sum_{\alpha \notin \mathcal{L}(T)} r_{\alpha} \prod_{\beta \in S(\alpha)}r_{\beta}) \text{ up to log factors }\\ 
& = \mathcal{O}(\mathcal{S}(T,r)),  \text{ up to log factors, }
\end{aligned}
\end{equation*}
where $\mathcal{S}(T,r)$ is the storage complexity of the tree tensor network $\mathcal{T}_r^T(V)$, $r = \{r_{\alpha}\}_{\alpha \in T}$, and $m = \{m_{\alpha}\}_{\alpha \in \mathcal{L}(T)}$.

\subsection{Heuristics used in practice}\label{subs:error_control_in_practice}
~
According to Lemma \ref{lem:bound_final_error_qo_proj}, the error of approximation is bounded in expectation by
\begin{equation*}
\mathbb{E}(\Vert u - u^{\star} \Vert^2) \le \sum_{\alpha \in T \setminus \{D\}} (2C_1)^{l(\alpha)}\mathbb{E}(\Vert \mathcal{Q}_{V_{\alpha}}u - \mathcal{P}_{\widehat{U}_{\alpha}^{\star}}\mathcal{Q}_{V_{\alpha}}u \Vert^2) + \sum_{\alpha \in \mathcal{L}(T)} \frac{1}{2} (2C_1)^{l(\alpha)+1} e^{\alpha,dis}_{m_{\alpha}}(u)^2.
\end{equation*}
The term $\mathbb{E}(\Vert \mathcal{Q}_{V_{\alpha}}u - \mathcal{P}_{\widehat{U}_{\alpha}^{\star}}\mathcal{Q}_{V_{\alpha}}u \Vert^2)$ includes the error due to the truncation and estimation of the $\alpha$-principal subspaces. The term $e^{\alpha,dis}_{m_{\alpha}}(u)$ is the discretization error made in the leaves, which comes from the introduction of finite-dimensional subspaces. These two contributions are amplified by constants depending on the boosted least-squares projection and the chosen tree. In the proposed adaptive strategies, if we want to obtain a certain precision $\varepsilon$ for the approximation, it is important to take these constants into account. Assuming that we want to reach a final error with precision $\varepsilon$, according to Lemma \ref{lem:bound_final_error_qo_proj} the following assumptions are a priori needed: 
\begin{itemize}
	\item For all $\alpha \in T \setminus D$, the term $\mathbb{E}(\Vert \mathcal{Q}_{V_{\alpha}}u - \mathcal{P}_{\widehat{U}_{\alpha}^{\star}}\mathcal{Q}_{V_{\alpha}}u \Vert^2)$ is controlled with Algorithm \ref{algo:adapt_pcs_precision} with prescribed tolerance $\varepsilon^2_{pca}$, i.e.
	$$ \mathbb{E}(\Vert \mathcal{Q}_{V_{\alpha}}u - \mathcal{P}_{\widehat{U}_{\alpha}^{\star}}\mathcal{Q}_{V_{\alpha}}u \Vert^2) \le \varepsilon_{pca}^2 : = \frac{\varepsilon^2}{(2C_1)^{l(\alpha)} (\#T-1) }, $$
	where $\#T$ is the number of nodes in the tree.	
\end{itemize}
\begin{itemize}
	\item For all $\alpha \in \mathcal{L}(T)$, the term $e^{\alpha,dis}_{m_{\alpha}}(u)$ is controlled with Algorithm \ref{algo:adapt_basis}, using for all $\alpha \in \mathcal{L}(T)$
	$$e^{\alpha,dis}_{m_{\alpha}}(u)^2 \le \varepsilon^2_{dis} := \frac{\varepsilon^2}{\frac{1}{2}(2C_1)^{l(\alpha)+1}d}.$$
	
\end{itemize}	 
Thus, under all these assumptions, we should get the desired accuracy for $\mathbb{E}(\Vert u - u^{\star} \Vert^2)$. In practice, the discretization errors $ e_{m_{\alpha}}^{\alpha,dis}(u)$ can be controlled by adapting the spaces $V_{\alpha}$, using the adaptive boosted optimal least-squares strategy described in \cite{HaNoPe2019}, for the construction of a sequence of boosted least-squares projections adapted to the sequence of spaces. For polynomial approximation, the sequence of nested subspaces is simply constructed by increasing the polynomial degree one by one. For wavelet approximation, it can be defined by increasing the resolution. The difficulty is that we have to perform this strategy for each sample $Q_{V_{\alpha}}u(\cdot, x_{\alpha^c}^k)$ of the function-valued random variable $Q_{V_{\alpha}}u(\cdot, X_{\alpha^c})$. More details about this strategy can be found in Appendix \ref{sec:adapt_basis}. However, for small values of $\varepsilon$, 
the constants $\varepsilon_{pca}$ and $\varepsilon_{dis}$ are likely to be very small. Indeed when choosing the boosted least-squares projection, the constant $C_1$ defined in Lemma \ref{lem:bound_final_error_qo_proj} may be high, particularly if the number of repetitions $M$ is high or the proportion $p_r$ of removed points is large, and the impact of a high value for $C_1$ will be all the more important as $l(\alpha)$ will be high (this will be particularly the case when using deep trees). Hence, very low values for $\varepsilon_{pca}$ and $\varepsilon_{dis}$ will result in very high rank specifications and the need to introduce high-dimensional spaces in the leaves. \\
Most often, such specifications tend to strongly underestimate the accuracy of the approximation. To better adapt the number of samples needed for a given error specification, the following heuristic choices are rather considered:
\begin{itemize}
	\item We replace the constant $C_1= 2(1+p_r(1-\delta)^{-1}(1-\eta^M)^{-1}M)$ by $C_1= 2(1+(1-\delta)^{-1}(1-\eta)^{-1})$, which corresponds to the boosted optimal weighted least-squares projection from Theorem \ref{th:qo_constant_bls}, with no repetition ($M=1$) and no subsampling ($p_r=1$). In \cite{HaNoPe2019}, we observed on all the examples (without noise) that these two choices give comparable accuracy for the error of approximation. This leads us to take the value $C_1 = (1-\delta^{-1})(1-\eta^{-1})$ even when there are repetitions and subsampling.
	\item When $l(\alpha) \ge 3$, we replace $l(\alpha)$ by $ 3$ in the expressions of $\varepsilon_{dis}$ and $\varepsilon_{pca}$. (In the examples from Sections  \ref{sec:numerical_examples_adaptive_approx_leaves} and \ref{sec:numerical_examples_adaptive_pca} the depth of the tree is lower or equal to $3$ so that this heuristic does not apply but we have observed on some examples that taking $l(\alpha) =1$ is not enough to control the precision). It only applies in examples from Section \ref{sec:numerical_examples_tree_adaptation}.
\end{itemize}

\subsection{Adaptive determination of the approximation spaces in the leaves}\label{sec:numerical_examples_adaptive_approx_leaves}
The discretization error made in the leaves depends on the approximation spaces we choose. In this section we focus on polynomial spaces and we use the adaptive strategy presented in Algorithm \ref{algo:adapt_basis} to select the polynomial degree $p$ that achieves the desired discretization error $\varepsilon_{dis}$.\\

To emphasize the importance of spaces adaptation, we consider the Henon-Heiles potential (see \cite{KressnerSteinlechnerUschmajew2014} for more details about this function) defined on $\mathcal{X} = \mathbb{R}^8$ ($d=8$) equipped with the standard Gaussian measure $\mu$: 
\begin{equation*}
u(x_1,\hdots, x_d) = \frac{1}{2} \sum_{i=1}^d x_i^2 + \sigma^{\star} \sum_{i=1}^{d-1}(x_ix_{i+1}^2-x_i^3) + \frac{\sigma^{\star}}{16} \sum_{i=1}^{d-1}(x_i^2 +x_{i+1}^2)^2,
\end{equation*}
with $\sigma^{\star} = 0.2$. For this function, there is no discretization error for $p \ge 4$, which allows a better interpretation of the results. Polynomial spaces $V_{\nu} = \mathbb{P}_p(\mathcal{X}_{\nu}),\  \nu \in D$, are then considered for the approximation. \\

\begin{table}[h!]\small 
	\centering                   
	\begin{tabular}{c|c|c|c|c|c|c}                                                                                                                       
		\rowcolor{gray!10!white} &	\multicolumn{4}{c|}{\textbf{Without basis adaptation}} & \multicolumn{2}{c}{\textbf{With basis adaptation}}\\
		&	\multicolumn{2}{c|}{$p=15$} & \multicolumn{2}{c|}{$p=4$} & \multicolumn{2}{c}{}\\
		\rowcolor{gray!10!white} &  $\mathcal{S}$ & $n$ &  $\mathcal{S}$ & $n$ & $\mathcal{S}$ & $n$ \\                                                                                                                                                                                                                                
		Interpolation & [761; 761] & [1097; 1097] & [431; 431] & [591; 591] & [461; 461] & [717; 717]\\
		\rowcolor{gray!10!white} Boosted Least-squares &$[761; 761]$ & [1108; 1109] & [431; 431] & [591; 591] &  [461; 461] &  [719; 720]\\                                                                                                                          
	\end{tabular}                      
	\caption{Comparison of the number of samples $n$ without and with basis adaptation necessary to get an approximation error $\varepsilon =10^{-14}$, using respectively interpolation with magic points and boosted least-squares projections. The $\alpha$-principal components are estimated with $z_{\alpha^c} = m_{\alpha}$.}
	\label{table:illustration_basis_adaptation}
\end{table}

Table \ref{table:illustration_basis_adaptation} compares the storage complexity $\mathcal{S}$, the number of evaluations $n$ in three cases. In the first two cases, there is no basis adaptation and we use respectively $p=15$ and $p=4$ such that in both cases there is no discretization error. In the third case, there is an adaptation of the basis (thanks to Algorithm \ref{algo:adapt_basis}) with maximal polynomial degree $p=15$. We observe for this example that the adaptive basis strategy is able to select a polynomial degree $p=5$ for each leaf of the tree, which is close to optimal, the overestimation of $p$ being explained by the choice of the stopping criterion of the algorithm (see Algorithm \ref{algo:adapt_basis}).

\subsection{Adaptive estimation of the $\alpha$-principal subspaces}\label{sec:numerical_examples_adaptive_pca}
To underline the importance of the adaptive estimation of the $\alpha$-principal subspaces, we now consider the following Anisotropic function, in dimension $d = 6:$
\begin{equation}
u(x)  = \frac{1}{(10+2x_1+x_3+2x_4-x_5)^2}
\end{equation}
defined on $\mathcal{X} = [-1,1]^d$ equipped with the uniform measure.\\
We also consider polynomial spaces $V_{\nu} = \mathbb{P}_{p}(\mathcal{X}_{\nu})$ for the approximation, with $p$ chosen adaptively. We construct the approximation in a tree-based tensor format with a balanced binary tree using Algorithm \ref{algo:constr_approx_TBT}.\\


\begin{table}[h!]
	\centering\small
	\begin{subtable}{.7\textwidth}
		\begin{tabular}{c|c|c|c|c}
			\rowcolor{gray!10!white}  $\log(\varepsilon)$ & $\log(\varepsilon(u^{\star}))$ & $\log(\sqrt{\mathbb{E}(\varepsilon(u^{\star})^2)})$ & $\mathcal{S}$ & $n$\\
			
			-2 & [-1.8; -0.8] & -1.4 & [66; 70] & [468; 492]  \\                                                                
			
			\rowcolor{gray!10!white}  -3 & [-2.1; -1.6] & -1.9 & [111; 132] & [586; 650] \\                                                                                                                                                                                                                                                    
			-4 & [-3.0; -2.3] & -2.7 & [160; 201] & [715; 833] \\                                                                                                                                                                                                                                                    
			\rowcolor{gray!10!white}  -5 & [-3.5; -3.1] & -3.3 & [250; 284] & [944; 1080] \\                                                           
			
			-6 & [-4.5; -3.2] & -3.8 & [343; 400] & [1194; 1449]\\ 
			\rowcolor{gray!10!white}	-7 & [-5.2; -4.1] & -4.7 & [590; 700] & [1597; 1999]\\                                                                                                                                                                                    
		\end{tabular}                                                                                                                     
		\caption{Without adaptive estimation of the principal components.}\label{subtable:Anisotropic_interpolation_withoutCC_nonadaptatif}          
	\end{subtable}
	\begin{subtable}{.7\textwidth}
		\begin{tabular}{c|c|c|c|c}
			\rowcolor{gray!10!white}  $\log(\varepsilon)$ & $\log(\varepsilon(u^{\star}))$ & $\log(\sqrt{\mathbb{E}(\varepsilon(u^{\star})^2)})$ & $\mathcal{S}$ & $n$\\
			
			-2  & [-1.7; -0.7] & -1.3 & [53; 74] & [180; 204] \\                                                                
			\rowcolor{gray!10!white}  -3 & [-2.3; -1.6] & -1.9 & [105; 153] & [241; 292] \\                                                                                                                                                                                                                                                    
			-4 & [-3.2; -1.8] & -2.5 & [175; 211] & [313; 361] \\                                                                                                                                                                                                                                                    
			\rowcolor{gray!10!white}  -5 & [-4.1; -3] & -3.6 & [251; 365] & [416; 533] \\                                                           
			-6 & [-4.7; -3.8] & -4.2 & [385; 490] & [545; 655] \\     
			\rowcolor{gray!10!white}	-7 & [-5.7; -4.1] & -4.8 & [680; 875] & [702; 895]\\                                                                                                                                                                                
		\end{tabular}                                                                                                                      
		\caption{With adaptive estimation of the principal components.}\label{subtable:Anisotropic_interpolation_withoutCC_adaptatif}          
	\end{subtable}
	\caption{Anisotropic function. Approximation using interpolation as projections with a prescribed tolerance for each $\alpha \in T$, $\varepsilon_{pca} = \varepsilon$ for the estimation of the principal components. Confidence intervals for relative error $\varepsilon(u^{\star})$, storage complexity $\mathcal{S}$, number of evaluations $n$.}\label{table:Anisotropic_interpolation_withoutCC_adaptatif}          
\end{table}

\begin{table}[h!]
	\centering\small
	\begin{subtable}{.7\textwidth}
		\begin{tabular}{c|c|c|c|c}
			\rowcolor{gray!10!white}  $\log(\varepsilon)$ & $\log(\varepsilon(u^{\star}))$ & $\log(\sqrt{\mathbb{E}(\varepsilon(u^{\star})^2)})$ & $\mathcal{S}$ & $n$\\
			-2 & [-3.3; -2.1] & -2.8 & [164; 185] & [708; 781] \\    \rowcolor{gray!10!white}  -3 & [-3.7; -2.9] & -3.3 & [202; 263] & [814; 1046] \\                                                  
			-4 & [-4.8; -3.4] & -4.1 & [333; 364] & [1137; 1348] \\                                                                                                                                                                                                                                                     
			\rowcolor{gray!10!white}  -5 & [-5.3; -4.1] & -4.7 & [450; 488] & [1707; 1852] \\                                             
			-6 & [-6.4; -4.6] & -5.5 & [566; 681] & [2012; 2657]  \\ 
			\rowcolor{gray!10!white}	-7 & [-6.7; -5.4] & -6 & [855; 965] & [2658; 3243]\\                                                                                                                                                                                                                           
		\end{tabular}                                                                                                                        
		\caption{Without adaptive estimation of the principal components.} 
		\label{subtable:Anisotropic_interpolation_withCC_NoNadaptatif}                                                                                                                                                      
	\end{subtable}
	\begin{subtable}{.7\textwidth}
		\begin{tabular}{c|c|c|c|c}
			\rowcolor{gray!10!white}  $\log(\varepsilon)$ & $\log(\varepsilon(u^{\star}))$ & $\log(\sqrt{\mathbb{E}(\varepsilon(u^{\star})^2)})$ & $\mathcal{S}$ & $n$\\
			-2 & [-3.2; -1.7] & -2.4 & [129; 205] & [269; 357] \\                                                                                                                                                                                                                                                    
			\rowcolor{gray!10!white}  -3 & [-3.9; -2.5] & -3.2 & [240; 391] & [395; 556] \\                                                  
			-4 & [-4.5; -3.3] & -3.9 & [399; 540] & [557; 717] \\                                                                                                                                                                                                                                                     
			\rowcolor{gray!10!white}  -5 & [-5.6; -4.3] & -4.9 & [526; 843] & [705; 1042] \\                                             
			-6 & [-6.3; -4.9] & -5.5 & [758; 1025] & [959; 1223] \\
			\rowcolor{gray!10!white}	-7 & [-7.3; -5.8] & -6.5 & [1070; 1461] & [1124; 1520]\\                                                                                                                                                                                                                           
		\end{tabular}                                                                                                                        
		\caption{With adaptive estimation of the principal components.}\label{subtable:Anisotropic_interpolation_withCC_adaptatif}                                                                                                                                                      
	\end{subtable}
	
	\caption{Anisotropic function. Approximation using interpolation as projections with a balanced binary tree, with a prescribed tolerance for each $\alpha \in T$, $\varepsilon_{pca}^2 = \frac{\varepsilon^2}{(2(1+(1-\delta)^{-1}(1-\eta)^{-1}))^{l(\alpha)}(\#T-1)}$ for the estimation of the principal components. Confidence intervals for relative error $\varepsilon(u^{\star})$, storage complexity $\mathcal{S}$, number of evaluations $n$.}\label{table:Anisotropic_interpolation_withCC_adaptatif}                                                                                                                                                      
\end{table}

\begin{table}[h!]
	\centering\small
	\begin{subtable}{.7\textwidth}
		\begin{tabular}{c|c|c|c|c}
			\rowcolor{gray!10!white}  $\log(\varepsilon)$ & $\log(\varepsilon(u^{\star}))$ & $\log(\sqrt{\mathbb{E}(\varepsilon(u^{\star})^2)})$ & $\mathcal{S}$ & $n$\\
			
			-2 & [-3.7; -2.3] & -3.2 & [213; 232] & [759; 819] \\                                                                                                                                                                                                                                                          
			\rowcolor{gray!10!white}  -3 & [-3.8; -2.8] & -3.3 & [253; 292] & [837; 944] \\                                                                                                                                                                                                                                  -4 & [-5.0; -3.4] & -4.2 & [321; 408] & [981; 1275] \\                                                             
			\rowcolor{gray!10!white}  -5 & [-5.1; -4.3] & -4.6 & [426; 507] & [1353; 1692] \\                                                                                                                                                                                                                              -6 & [-5.8; -4.9] & -5.4 & [551; 656] & [1823; 2329]\\                                                   
			\rowcolor{gray!10!white}  -7 & [-6.7; -5.3] & -6.0 & [735; 875] & [2851; 3791]\\                                                                                                                                                                                                                             
		\end{tabular}                                                                                                                 
		\caption{Without adaptive estimation of the principal components.} \label{subtable:Anisotropic_bls_Nonadaptatif}                                                                                                                                            
	\end{subtable}
	\begin{subtable}{.7\textwidth}
		\begin{tabular}{c|c|c|c|c}
			\rowcolor{gray!10!white}  $\log(\varepsilon)$ & $\log(\varepsilon(u^{\star}))$ & $\log(\sqrt{\mathbb{E}(\varepsilon(u^{\star})^2)})$ & $\mathcal{S}$ & $n$\\
			
			-2 & [-3.6; -2.3] & -3 & [193; 270] & [328; 403] \\                                                                                                                                                                                                                                                          
			\rowcolor{gray!10!white}  -3 & [-5.0; -3.3] & -4.1 & [309; 430] & [455; 579] \\                                                                                                                                                                                                                                  -4 & [-4.9; -3.8] & -4.4 & [385; 531] & [534; 697] \\                                                             
			\rowcolor{gray!10!white}  -5 & [-6.2; -4.4] & -5.3 & [588; 805] & [751; 985] \\                                                                                                                                                                                                                              -6 & [-6.7; -5.5] & -6.1 & [827; 1268] & [1028; 1503] \\                                                   
			\rowcolor{gray!10!white}  -7 & [-7.7; -6.2] & -7.0 & [1203; 1861] & [1463; 2230] \\                                                                                                                                                                                                                             
		\end{tabular}                                                                                                                 
		\caption{With adaptive estimation of the principal components.} \label{subtable:Anisotropic_bls_adaptatif}                                                                                                                                            
	\end{subtable}
	\caption{Anisotropic function. Approximation using boosted least-squares as projections with a balanced binary tree, with a prescribed tolerance for each $\alpha \in T$ $\varepsilon_{pca}^2 = \frac{\varepsilon^2}{(2(1+(1-\delta)^{-1}(1-\eta)^{-1}))^{l(\alpha)}(\#T-1)}$ for the estimation of the principal components. Confidence intervals for relative error $\varepsilon(u^{\star})$, storage complexity $\mathcal{S}$, number of evaluations $n$.}    
	\label{table:Anisotropic_bls_adaptatif}                                                                                                                                            
\end{table} 

The results are summarized in Tables \ref{table:Anisotropic_interpolation_withoutCC_adaptatif}, \ref{table:Anisotropic_interpolation_withCC_adaptatif} and \ref{subtable:Anisotropic_bls_adaptatif}.
In Table \ref{table:Anisotropic_interpolation_withoutCC_adaptatif}, we first note that choosing $\varepsilon_{pca} = \varepsilon$ and using interpolation does not provide an approximation error that is lower than $\varepsilon$. In this case, we also see that the adaptive strategy for the estimation of the principal components performs better in average (in the sense that the obtained approximation error is lower) than in the non adaptive case, and this with fewer evaluations. Table \ref{table:Anisotropic_interpolation_withCC_adaptatif} then shows that for each $\alpha \in T$, choosing $\varepsilon_{pca}^2 = \frac{\varepsilon^2}{(2(1+(1-\delta)^{-1}(1-\eta)^{-1}))^{l(\alpha)}(\#T-1)}$ and using interpolation does not provide a controlled approximation error for the small values of $\varepsilon$ (i.e $\log(\varepsilon)$ lower than $-4$), both when the principal components are adaptively determined or not. However, using the adaptive strategy strongly reduces the number of samples. Finally, Table \ref{table:Anisotropic_bls_adaptatif} shows that for each $\alpha \in T$, choosing $\varepsilon_{pca}^2 = \frac{\varepsilon^2}{(2(1+(1-\delta)^{-1}(1-\eta)^{-1}))^{l(\alpha)}(\#T-1)}$ and using the boosted optimal weighted least-squares projection (even with $M=100$ repetitions and subsampling) provides this time a controlled approximation error. What is also very interesting is that the adaptive strategy for principal components estimation strongly reduces the number of samples necessary to reach the desired accuracy.

\subsection{Adaptive approximation of the tree}
\label{sec:numerical_examples_tree_adaptation}
In this section, we want to illustrate the efficiency of the proposed strategy for tree optimization. To this end, we focus on two numerical examples, for which the optimal dimension tree is known: 
\begin{itemize}
	\item A sum of bivariate functions with separated variables defined by
	\begin{equation}
	\label{eq:sum_of_bivariate_functions}
	u(x) = g(x_1,x_2) + g(x_3,x_4)+ \hdots +  g(x_{d-1},x_{d}), 
	\end{equation}
	where we consider $\mathcal{X} = [-1,1]^d$, equipped with the uniform measure and $g(x_{\nu},x_{\nu+1}) = \sum_{i=0}^3 x_{\nu}^i x_{\nu + 1}^i$.
	\item A sum of trivariate functions with interlaced variables defined by
	\begin{equation}
	\label{eq:another_sum_of_bivariate_functions}
	u(x) = g(x_1,x_2,x_3) + g(x_2,x_3,x_4)+ \hdots +  g(x_{d-3},x_{d-2},x_{d-1}) +  g(x_{d-2},x_{d-1},x_{d}), 
	\end{equation}
	where we consider $\mathcal{X} = [-1,1]^d$, equipped with the uniform measure and $g(x_{\nu},x_{\nu+1},x_{\nu+2}) = \sum_{i=0}^2 x_{\nu}^i x_{\nu + 1}^i x_{\nu + 2}^i$. We consider $V = \bigotimes_{\nu =1}^d \mathbb{P}_3(\mathcal{X}^{\nu})$, so that there is no discretization error.
\end{itemize}
We denote by \textbf{s-LO} the proposed leaves-to-root strategy with stochastic local optimizations and by \textbf{bg-LO} the strategy proposed in \cite{BallaniGrasedyck2014} where the optimization criterion is a rank ratio. These strategies are also compared with a random tree \textbf{RT} (with arity two) and a random balanced tree referred to as \textbf{RBT}. In these two last cases, the overall number of evaluations of the function is dedicated to the computation of the approximation, such that $n_{optim} =0$ and $n = n_{total}$.\\

\begin{table}[h!]     \small                                                                                 
	\centering                                                                                            
	\begin{tabular}{c|c|c|c|c}                                                                        
		\rowcolor{gray!10!white}	& $\log(\varepsilon(u^{\star}))$ &  $\mathcal{S}$ & $n$ & $n_{total}$  \\                                                                                                                                          
		& [$q_{10}$; $q_{50}$; $q_{90}$] & [$q_{10}$; $q_{50}$; $q_{90}$] & [$q_{10}$; $q_{50}$; $q_{90}$] & [$q_{10}$; $q_{50}$; $q_{90}$]\\                                                                                                                                      
		\rowcolor{gray!10!white}\textbf{s-LO}  &  [-14.9; -14.7; -14.0] & [340; 529; 1360] & [468; 689; 1552] & [1243; 1699; 2752]  \\        
		\textbf{bg-LO} &  [-15; -14.7; -14.5] & [354; 376; 445] & [485; 512; 574] & [3239; 3372; 3611]  \\                                       
		\rowcolor{gray!10!white}	\textbf{RBT}  &  [-14.4; -14.3; -13.9] & [696; 925; 2198] & [858; 1150; 2432] & [858; 1150; 2432]  \\     
		\textbf{RT} & [-14.6; -14.3; -14.1] & [971; 1763; 2471] & [1166; 1987; 2745] & [1166; 1987; 2745] \\     
		
	\end{tabular}                                                                                
	\caption{Sum of bivariate functions defined by \Cref{eq:sum_of_bivariate_functions} with $d=8$. Approximation with a prescribed tolerance $\varepsilon = 10^{-14}, \varepsilon_c = 10^{-2}$. $q_{10}, q_{50}, q_{90}$ are the $10^{th}, 50^{th}$ and $90^{th}$ quantiles for relative error $\log(\varepsilon(u^{\star}))$, number of evaluations $n$, for the \textbf{s-LO} strategy $\gamma = 6$ and the maximal number of iterations $n_P = 2d$.} 
	\label{table:optim_sum_of_bivariate_d8}
\end{table}

\begin{table}[h!]    \small                                                                                 
	\centering                                                                                            
	\begin{tabular}{c|c|c|c|c}                                                                        
		\rowcolor{gray!10!white}	& $\log(\varepsilon(u^{\star}))$ &  $\mathcal{S}$ & $n$ & $n_{total}$  \\                                                                                                                                          
		& [$q_{10}$; $q_{50}$; $q_{90}$] & [$q_{10}$; $q_{50}$; $q_{90}$] & [$q_{10}$; $q_{50}$; $q_{90}$] & [$q_{10}$; $q_{50}$; $q_{90}$]\\                                                                                                                                                  
		\rowcolor{gray!10!white}	\textbf{s-LO}  &  [-14.4; -14.1; -13.7] & [949; 2011; 4167] & [1259; 2394; 4717] & [3874; 4979; 7802]  \\   
		\textbf{bg-LO} & [-14.7; -14.5; -14.2] & [692; 729; 882] & [956; 993; 1165] & [11336; 11869; 12773] \\
		\rowcolor{gray!10!white}	\textbf{RBT}  & [-14.0; -13.8; -13.2] & [7476; 10888; 14837] & [8148; 11657; 15651] & [8148; 11657; 15651]  \\                 
		\textbf{RT} & [-14.0; -13.7; -12.7] & [5033; 11320; 36456] & [5600; 12154; 37635] & [5600; 12154; 37635] \\                                                                                                          
	\end{tabular}                                                                                
	\caption{Sum of bivariate functions defined by \Cref{eq:sum_of_bivariate_functions} with $d=16$. Approximation with a prescribed tolerance $\varepsilon = 10^{-14}, \varepsilon_c = 10^{-2}$. $q_{10}, q_{50}, q_{90}$ are the $10^{th}, 50^{th}$ and $90^{th}$ quantiles for relative error $\log(\varepsilon(u^{\star}))$, number of evaluations $n$, for the \textbf{s-LO} strategy $\gamma = 6$ and the maximal number of iterations $n_P = 2d$.} 
	\label{table:optim_sum_of_bivariate_d16}
\end{table}

\begin{table}[h!]    \small                                                                                 
	\centering                                                                                            
	\begin{tabular}{c|c|c|c|c}                                                                        
		\rowcolor{gray!10!white}	& $\log(\varepsilon(u^{\star}))$ &  $\mathcal{S}$ & $n$ & $n_{total}$  \\                                                                                                                                          
		& [$q_{10}$; $q_{50}$; $q_{90}$] & [$q_{10}$; $q_{50}$; $q_{90}$] & [$q_{10}$; $q_{50}$; $q_{90}$] & [$q_{10}$; $q_{50}$; $q_{90}$]\\   
		
		\rowcolor{gray!10!white}\textbf{s-LO} & [-14.2; -13.9; -13.6] & [2395; 5773; 8183] & [3035; 6685; 9177] & [7225; 12185; 15842] \\                                                                                                   
		\textbf{bg-LO} & [-14.3; -13.6; -13] & [1940; 5565; 11961] & [2505; 6374; 13041] & [25855; 31594; 38246] \\         
		
		\rowcolor{gray!10!white}	\textbf{RBT} & [-13.6; -13; -12.5] & [16001; 22079; 46982] & [17305; 23634; 48815] & [17305; 23634; 48815] \\         
		
		\textbf{RT} &[-13.7; -12.9; -12.2] & [15100; 22745; 32182] & [16418; 24269; 33793] & [16418; 24269; 33793] \\     
	\end{tabular}                                                                                
	\caption{Sum of bivariate functions defined by \Cref{eq:sum_of_bivariate_functions} with $d=24$. Approximation with a prescribed tolerance $\varepsilon = 10^{-13}, \varepsilon_c = 10^{-2}$. $q_{10}, q_{50}, q_{90}$ are the $10^{th}, 50^{th}$ and $90^{th}$ quantiles for relative error $\log(\varepsilon(u^{\star}))$, number of evaluations $n$, for the \textbf{s-LO} strategy $\gamma = 6$ and the maximal number of iterations $n_P = 2d$.} 
	\label{table:optim_sum_of_bivariate_d24}
\end{table}

The results associated with the bivariate function are summarized in Tables \ref{table:optim_sum_of_bivariate_d8}, \ref{table:optim_sum_of_bivariate_d16} and \ref{table:optim_sum_of_bivariate_d24}.
Focusing on $d=8$, Table \ref{table:optim_sum_of_bivariate_d8} shows that both optimization strategies decrease the number of evaluations $n$ necessary to compute the approximation with precision $\varepsilon$ compared to a random tree \textbf{RT} or a random balanced tree \textbf{RBT}. However the total number of evaluations (including the evaluations used for the tree search) is in average slightly greater than the cost of a random tree \textbf{RT}. This is due to the fact that the input space dimension is small and the $\alpha$-ranks (even chosen randomly) remain moderate. The results are quite different if we are interested in higher dimensional functions. For instance, if we focus on $d=16$,
Table \ref{table:optim_sum_of_bivariate_d16} shows that both optimization strategies decrease the number of evaluations $n$ necessary to compute the approximation with precision $\varepsilon$ compared to a random tree \textbf{RT} or a random balanced tree \textbf{RBT}. But this time, for all the optimization strategies, the $90^{th}$ quantile of the total number of evaluations $n_{total}$ is lower than the cost of a random tree \textbf{RT} or a random balanced tree \textbf{RBT}. In this case, the \textbf{s-LO} strategy is the most efficient method as it decreases the three quantiles compared to the random trees. The interest of using optimization strategies when the dimension increases is all the more underlined when choosing $d=24$, as it is done in Table \ref{table:optim_sum_of_bivariate_d24}. In that case, the number of evaluations $n$ necessary to compute the approximation is once again strongly reduced. And for both optimization methods the $90^{th}$ quantile of the number of total evaluations is much lower than with a random tree \textbf{RT} or a random balanced tree \textbf{RBT}. We also notice that for this example, the \textbf{bg-LO} method recovers the best tree, but the additional cost used to evaluate the $\alpha$-ranks for the optimization, which appears in $n_{total}$, is this time not competitive for the $10^{th}$ and $50^{th}$ quantiles. On the contrary, the local stochastic strategy we propose performs particularly well as the $90^{th}$ quantile of the number of evaluations is lower than the $10^{th}$ quantile of the number of evaluations necessary for a random tree \textbf{RT} and \textbf{RBT}.\\

\begin{table}[h!]     \small                                                                                 
	\centering                                                                                            
	\begin{tabular}{c|c|c|c|c}                                                                        
		\rowcolor{gray!10!white}	& $\log(\varepsilon(u^{\star}))$ & $\mathcal{S}$ & $n$ & $n_{total}$  \\                                                                                                                                          
		& [$q_{10}$; $q_{50}$; $q_{90}$] & [$q_{10}$; $q_{50}$; $q_{90}$] & [$q_{10}$; $q_{50}$; $q_{90}$] & [$q_{10}$; $q_{50}$; $q_{90}$]\\                                                                                                                                                   
		\rowcolor{gray!10!white} \textbf{s-LO}  & [-14.3; -13.9; -13.7] & [7733; 10812; 12189] & [8376; 11574; 12977] & [12211; 15487; 16987] \\                       
		\textbf{bg-LO} & [-14.3; -14.1; -14] & [2204; 8267; 13762] & [2665; 9076; 14755] & [26450; 34031; 39490] \\                            
		\rowcolor{gray!10!white} 	\textbf{RBT}  & [-14; -13.8; -13.4] & [10647; 12420; 19681] & [11392; 13235; 20673] & [11392; 13235; 20673] \\             
		\textbf{RT} &  [-14.2; -13.8; -13.3] & [9063; 12262; 23263] & [9817; 13164; 24406] & [9817; 13164; 24406] \\   
	\end{tabular}                                                                                
	\caption{Sum of trivariate functions defined by \Cref{eq:another_sum_of_bivariate_functions} with $d=19$. Approximation with a prescribed tolerance $\varepsilon = 10^{-14}, \varepsilon_c = 10^{-2}$. $q_{10}, q_{50}, q_{90}$ are the $10^{th}, 50^{th}$ and $90^{th}$ quantiles for relative error $\log(\varepsilon(u^{\star}))$, number of evaluations $n$, for the \textbf{s-LO} strategy $\gamma = 6$ and the maximal number of iterations $n_P = 2d$.} 
	\label{table:optim_HH_d8}
\end{table}

The results associated with the sum of trivariate functions are summarized in Table \ref{table:optim_HH_d8}. In line with the results associated with the sum of bivariate functions, we notice again in this table that the different optimization strategies allow us to decrease the number of evaluations $n$ necessary to compute the approximation with precision $\varepsilon$ compared to a random tree \textbf{RT} or a random balanced tree \textbf{RBT}. But the total number of evaluations for the deterministic optimization strategies is higher than the cost of a random tree. This is due to the fact that whatever the tree is, the $\alpha$-ranks (even chosen randomly) remain moderate, such that the additional cost due to the $\alpha$-ranks estimations may not be useful. However with the stochastic strategy, the total number of evaluations is decreased compared to a random tree. 

\section{Conclusions}
In this paper, we have proposed an algorithm to construct the approximation of a function $u$ in tree tensor format $\mathcal{T}_r^T(V)$ with $V = \bigotimes_{\nu=1}^d V_{\nu}$ some background approximation space possibly selected adaptively. Using adaptive strategies for the control of the discretization error, the control of the $\alpha$-ranks and the estimation of the principal components we are able to provide a controlled approximation of the function $u$, assuming we have a sufficiently high number of evaluations. The theoritical criteria used to control the approximation appear to be very pessimistic for two reasons:
\begin{itemize}
	\item As underlined in \cite{HaNoPe2019}, the constant of quasi-optimality $C_1$ of the boosted least-squares projection is loose compared to what we observe in practice.
	\item The proof of Theorem \ref{th:final_err_bound} leads to a bound with the constant $C_1$ to the power of the depth of the tree. 
\end{itemize}
On the studied examples, these theoretical bounds turn out to be pessimistic. However, as this bound has been etablished for any function from $L^2_{\mu}$, some functions may indeed verify this bound (we have just not found them yet).\\

~
In this work, we proposed several optimization strategies for choosing a dimension partition tree $T$, which is adapted to the function we want to approximate, in the sense that the ranks of the approximation remain small to get a certain accuracy. The deterministic strategy from the literature explore a large number of trees and are able to recover really good trees to reach low ranks. However this exploration is often too expensive compared to the number of evaluations necessary for the approximation of the function. In the presented cases, using these strategies may sometimes lead to an overall number of samples which is smaller to the one required by random trees with high $\alpha$-ranks. But this is not always the case, and the number of evaluations associated to the selection of the dimension tree may be very high.\\

The presented stochastic strategy (with a few exploration steps) is more competitive regarding the number of evaluations for the estimations of the ranks. However this strategies involves several numerical (and heuristics) parameters, which need to be tuned. Furthermore, if the choices made for these examples are working relatively well we do not claim that this will be efficient for any  function. 

\bibliographystyle{plain}
\bibliography{references}
\begin{appendix}

	\section*{Appendix}
	
	\section{Estimation of the $\alpha$-principal subspaces}\label{appendix:estimation_pcs_approx}
	We here present the practical aspects for estimating the $\alpha$-principal subspaces of $Q_{V_{\alpha}}u$. 
	Let $(\varphi_j^{\alpha})_{j=1}^{m_{\alpha}}$ be an orthonormal basis of $V_{\alpha}$. Then $Q_{V_{\alpha}}u(\cdot, x_{\alpha}^l)$ can be written $Q_{V_{\alpha}}u(\cdot, x_{\alpha^c}^l) = \sum_{j=1}^{m_{\alpha}} a^{jl}_{\alpha} \varphi^{\alpha}_j(\cdot)$
	where the coefficients $a^{jl}_{\alpha}$ depend on the samples $\{x_{\alpha}^i\}_{i=1}^{z_{\alpha}}$ in $\mathcal{X}_{\alpha}$ used to define the projection $Q_{V_{\alpha}}$. Therefore, solving the Equation  (\ref{eq:empirical_pcs}) requires evaluating the function $u$ on a product grid $\{(x^i_{\alpha},x^l_{\alpha^c}) : 1 \le i \le z_{\alpha}, 1 \le l \le z_{\alpha^c}\}$, where the samples $(x^i_{\alpha},x^l_{\alpha^c})$ are not i.i.d..
	We denote by $\bm{A}^{\alpha} \in \mathbb{R}^{m_{\alpha} \times z_{\alpha^c}}$ the matrix formed with the coefficients $(a^{il}_{\alpha})$.
	The truncated singular value decomposition of $\bm{A}^{\alpha}$ is $
	\bm{A}^{\alpha}_{r_{\alpha}} = \sum_{k=1}^{r_{\alpha}} \sigma^k_{\alpha} \bm{v}^k_{\alpha}(\bm{v}^k_{\alpha^c})^T$
	where $\bm{v}^k_{\alpha} = (v^{ki}_{\alpha})_{1 \le i \le m_{\alpha}} \in \mathbb{R}^{m_{\alpha}}$ and $\bm{v}^k_{\alpha^c} = (v^{kl}_{\alpha^c})_{1 \le l \le z_{\alpha^c}} \in \mathbb{R}^{z_{\alpha^c}}$, and the $\sigma^1_{\alpha} \ge \sigma^{k}_{\alpha} \ge \hdots \sigma^{r_{\alpha}}_{\alpha}$ are the singular values, which are assumed to be sorted in decreasing order.\\
	The solution of Equation (\ref{eq:empirical_pcs}) is the subspace spanned by the functions 
	$v^{\alpha}_k(\cdot) = \sum_{j=1}^{m_{\alpha}} v^{kj}_{\alpha} \varphi_j^{\alpha}(\cdot)$,  for $\ 1 \le k \le r_{\alpha}.$
	Letting $\bm{V}^{\alpha}_{r_{\alpha}} = (\bm{v}_{\alpha}^1, \hdots, \bm{v}_{\alpha}^{r_{\alpha}})$, we have 
	\begin{equation*}
	\sum_{l=1}^{z_{\alpha^c}} \Vert Q_{V_{\alpha}}u(\cdot, x^l_{\alpha^c}) - P_{\widehat{U}_{\alpha}^{\star}}Q_{V_{\alpha}}u(\cdot,x^l_{\alpha^c})\Vert_{L^2_{\mu_{\alpha}}}^2 = \Vert \bm{A}^{\alpha}_{r_{\alpha}} - \bm{V}^{\alpha}_{r_{\alpha}}(\bm{V}^{\alpha}_{r_{\alpha}})^T\bm{A}^{\alpha}_{r_{\alpha}} \Vert_F^2 = \sum_{k > r_{\alpha}} (\sigma^k_{\alpha})^2.
	\end{equation*}
	The rank $r_{\alpha}$ can be chosen such that $\sum_{k > r_{\alpha}} (\sigma^k_{\alpha})^2 \le \varepsilon^2 \sum_{k \ge  1} (\sigma^k_{\alpha})^2$ implying that
	\begin{equation*}
	\frac{1}{z_{\alpha^c}} \sum_{l=1}^{z_{\alpha^c}} \Vert Q_{V_{\alpha}}u(\cdot, x^l_{\alpha^c}) - P_{\widehat{U}_{\alpha}^{\star}}Q_{V_{\alpha}}u(\cdot,x^l_{\alpha^c})\Vert_{L^2_{\mu_{\alpha}}}^2 \le \frac{\varepsilon^2}{z_{\alpha^c}} \sum_{l=1}^{z_{\alpha^c}} \Vert Q_{V_{\alpha}}u(\cdot, x^l_{\alpha^c}) \Vert_{L^2_{\mu_{\alpha}}}^2.
	\end{equation*}
	We describe in Algorithm \ref{algo:adapt_pcs_precision} a procedure that adapts both the rank $r_\alpha$ and the number of samples $z_{\alpha^c}$ to estimate a subspace $\widehat{U}_{\alpha}^{\star}$ that yields a reconstruction error with a prescribed precision.

	\begin{algorithm}[h!]
		\caption{Adaptive algorithm for the estimation of the $\alpha$-principal components of a function-valued random variable $X_{\alpha^c} \mapsto Q_{V_{\alpha}}u(\cdot, X_{\alpha^c})$ with prescribed tolerance $\varepsilon$ }\label{algo:adapt_pcs_precision}
		\begin{algorithmic}
			\STATE \hspace{-0.4cm}\textbf{Inputs:} desired tolerance $\varepsilon$, random variable $u(\cdot, X_{\alpha^c})$, approximation space $V_{\alpha}$, $(\varphi_j)_{j=1}^{m_{\alpha}}$ an orthonormal basis of $V_{\alpha}$, oblique projection $Q_{V_{\alpha}}$ and sampling factor $k_{PCA}$.
			\STATE \hspace{-0.4cm}\textbf{Outputs:} $\bm{V}^{\alpha}_r$ matrix of singular vectors of $\bm{A}^{\alpha}$
			\STATE Set $z_{\alpha^c} =1$
			\STATE Compute the vector $\bm{A}^{\alpha}$ corresponding to the coefficients of one realization $Q_{V_{\alpha}}u(\cdot,X_{\alpha^c})$ in the orthonormal basis of $V_{\alpha}$.
			\STATE Set $\mathcal{E} = \infty$
			\WHILE{$\mathcal{E} > \varepsilon$ and $z_{\alpha^c} \le k_{PCA}\dim(V_{\alpha})$}
			\STATE Update $z_{\alpha^c} \leftarrow z_{\alpha^c} + 1$ 
			\STATE Update $\bm{A}^{\alpha} = [\bm{A}^{\alpha}, \bm{a}^{\alpha}]$, with $\bm{a}^{\alpha}$ the vector corresponding to the coefficients of one realization of $Q_{V_{\alpha}}u(\cdot,X_{\alpha^c})$ in the orthonormal basis of $V_{\alpha}$.
			\STATE Set $r =0$
			\WHILE{$\mathcal{E} > \varepsilon$ or $r\le z_{\alpha^c}$}
			\STATE Update $r \leftarrow r+1$
			\STATE Compute $\mathcal{E}$ the leave-one-out cross validation error,
			\FOR{$l=1, \hdots, z_{\alpha^c}$}
			\STATE Determine the matrix $\bm{V}^{\alpha}_{\setminus l,r}$ of $r$ main left singular vectors of $\bm{A}^{\alpha}_{\setminus l}$, which is $\bm{A}^{\alpha}$ without its $l^{th}$ column.
			\ENDFOR
			\STATE Set
			\begin{equation}
			\mathcal{E} = \frac{  \sum_{l=1}^{z_{\alpha^c}} \Vert \bm{A}^{\alpha}_l - \bm{V}^{\alpha}_{\setminus l,r}(\bm{V}^{\alpha}_{\setminus l,r})^T \bm{A}^{\alpha}_l \Vert_2^2 }{  \sum_{l=1}^{z_{\alpha^c}} \Vert \bm{A}^{\alpha}_l \Vert_2^2 }.
			\end{equation}
			\ENDWHILE
			\ENDWHILE
			\STATE Determine the matrix $\bm{V}^{\alpha}_r$ of $r$ left singular vectors of $\bm{A}^{\alpha}$.
		\end{algorithmic}
	\end{algorithm}
	
	\section{Adaptive determination of feature spaces $V_\alpha$}\label{sec:adapt_basis}
	We present in Algorithm \ref{algo:adapt_basis} an adaptive procedure for the selection of approximation (or feature) spaces $V_\alpha$, $\alpha \in \mathcal{L}(T)$, from a sequence of candidate spaces.
	
	\begin{algorithm}[h!]
		\caption{Algorithm for adaptive approximation of $u(\cdot, x_{\alpha^c}^k)$.}\label{algo:adapt_basis}
		\begin{algorithmic}
			\STATE \hspace{-0.4cm}\textbf{Inputs:} desired tolerance $\varepsilon_{dis}$, a sequence of nested approximation spaces $(V_{\alpha}^j)_{j \ge 1}$.
			\STATE \hspace{-0.4cm}\textbf{Outputs:} $\bm{a}^{\alpha}$ the vector corresponding to the coefficients of the $k^{th}$ realization of $Q_{V_{\alpha}^j}u(\cdot,x_{\alpha^c}^k)$ in the orthonormal basis of $V_{\alpha}^{j}$, with $j$ depending on the desired tolerance $\varepsilon$.
			\STATE Set $\mathcal{E} = \infty$
			\WHILE{$\mathcal{E} > \varepsilon_{dis}$}
			\STATE Compute the coefficients $\bm{a}^{\alpha}$ of the boosted least-squares projection $Q_{V_{\alpha}^j}u(\cdot,x_{\alpha^c}^k)$ in the orthonormal basis of $V_{\alpha}^j$
			\STATE Set $\mathcal{E} = \frac{a^{\alpha}_p}{\sqrt{\sum_{i=1}^p (a^{\alpha}_i)^2}}$
			\STATE Set $j = j+1$
			\ENDWHILE
		\end{algorithmic}
	\end{algorithm}
	
	\section{Explicit expressions of the optimal measure}
	\label{app:explicitExpression}
	
	The boosted optimal weighted least-squares projection, on which our learning algorithm is based, relies on the generation of samples associated with the optimal measure defined by Eq. (\ref{eq:optimal_measure}). In this section, we make explicit the expression of this optimal measure, in the case where $\alpha$ is a leaf of the tree or an interior node.
	\begin{itemize}
		\item When $\alpha \in \mathcal{L}(T)$ is a leaf node, $V_{\alpha}$ is a given approximation space of univariate functions, such that sampling only implies one-dimensional distributions. One can then rely on standard simulation methods such as rejection sampling, inverse transform sampling or slice sampling techniques, see \cite{Devroye1985}.
		\item When $\alpha \notin \mathcal{L}(T)$, $V_{\alpha} = \otimes_{\beta \in S} \widehat{U}_{\beta}^{\star}$. For each $\beta$, we let $\{\psi_{k_{\beta}}^{\beta} \}_{k_{\beta} = 1}^{r_{\beta}}$ be a basis of $\widehat{U}_{\beta}^{\star}$. The product basis of $V_{\alpha}$ is denoted by $\{ \varphi_{i_{\alpha}}^{\alpha}\}_{i_{\alpha}=1}^{m_{\alpha}}$, where $m_{\alpha} = \prod_{\beta \in S(\alpha)}r_{\beta}$ and for $i_{\alpha} =1, \hdots, m_{\alpha}$,
		$\varphi_{i_{\alpha}}^{\alpha}(x_{\alpha}) = \prod_{\beta \in S(\alpha)} \psi_{k_{\beta}}^{\beta} (x_{\beta}), \text{ for } i_{\alpha} \equiv (k_{\beta})_{\beta \in S(\alpha)}.$
		The sampling measure given by Equation \Cref{eq:optimal_measure} is such that
		\begin{equation*}
		w^{\alpha}(x_{\alpha})^{-1} = \frac{1}{m_{\alpha}} \sum_{i_{\alpha}=1}^{m_{\alpha}} \varphi_{i_{\alpha}}^{\alpha}(x_{\alpha})^2  =  \prod_{\beta \in S(\alpha)} \frac{1}{r_{\beta}} \sum_{1 \le k_{\beta} \le r_{\beta}}  \psi_{k_{\beta}}^{\beta} (x_{\beta})^2
		\end{equation*}
		and using the product structure of $\mu_{\alpha}$, we have
		\begin{equation*}
		d\rho_{\alpha}(x_{\alpha}) = \prod_{\beta \in S(\alpha)} d\rho_{\beta}(x_{\beta}) \text{ with } d\rho_{\beta}(x_{\beta}) = \frac{1}{r_{\beta}}\sum_{k_{\beta} =1}^{r_{\beta}} \psi_{k_{\beta}}^{\beta} (x_{\beta})^2 d\mu_{\beta}(x_{\beta}).
		\end{equation*}
		As for each $\beta \in S(\alpha)$, $d\rho_{\beta}(x_{\beta})$ can be written in tree tensor networks format, its marginal distributions can be efficiently computed. Then sampling from $\rho_{\beta}$ can be efficiently done through sequential sampling. The interested reader is referred to \cite{Haberstich2020} for some implementation details.	\end{itemize}

	\section{Estimation of the $\alpha$-ranks of a function $u$ to perform tree adaptation.}
	\label{appendix:alpha_rank_estimation}
	We present here the algorithm that estimates $\alpha$-ranks for tree adaptation. The strategy is described in Section \ref{subs:eps_rank_estimation}.
	\begin{algorithm}[h!]
		\caption{ Determination of the ranks $r_{\alpha}({\varepsilon})$ of a function $u$ }\label{algo:eval_eps_ranks}
		\begin{algorithmic}
			\STATE \hspace{-0.4cm} \textbf{Inputs:} coarse tolerance $\varepsilon$, function $u$, tuple $\alpha$, product measure $\mu$, $n_{\alpha}$ and $n_{\alpha^c}$
			\STATE \hspace{-0.4cm} \textbf{Outputs:} $r_{\alpha}$ and $cost = z_{\alpha} z_{\alpha^c}$
			\STATE Generate $z_{\alpha}$ i.i.d samples $\{x_{\alpha}^i\}_{i=1}^{z_{\alpha}}$ from the measure $\mu_{\alpha}$
			\STATE Generate $z_{\alpha^c}=1$ sample $\{x_{\alpha^c}^l\}_{l=1}^{z_{\alpha^c}}$ from the measure $\mu_{\alpha^c}$
			\STATE Evaluate the function $u$ on the grid $\{ (x_{\alpha}^i, x_{\alpha^c}^j)  : 1 \le i \le z_{\alpha}, 1 \le l \le z_{\alpha^c}\}$ and set $\bm{B}^{\alpha} = (u(x_{\alpha}^i, x_{\alpha^c}^l))$
			\STATE Set $r = 0$ and $\mathcal{E}(r) = \infty$
			\WHILE{$\mathcal{E}(r) > \varepsilon_c$ and $z_{\alpha^c} \le n_{\alpha^c}$}
			\STATE Set $z_{\alpha^c} \leftarrow z_{\alpha^c} + 1$ 
			\STATE Sample $\{x_{\alpha^c}^{z_{\alpha^c}}\}$ from the measure $\mu_{\alpha^c}$
			\STATE Update $\bm{B}^{\alpha} = [\bm{B}^{\alpha}, \bm{b}^{\alpha}]$, with $\bm{b}^{\alpha}$ the vector corresponding to the evaluations of $u$ on the grid $\{ (x_{\alpha}^i, x_{\alpha^c}^{z_{\alpha^c}})  : 1 \le l \le z_{\alpha} \}$
			\WHILE{($\mathcal{E}(r) > \varepsilon$ and $r < \min(n_{\alpha},z_{\alpha^c})$)}
			\STATE Set $r \leftarrow r +1$
			\FOR{$l=1, \hdots, z_{\alpha^c}$}
			\STATE Determine the matrix $\bm{V}^{\alpha}_{\setminus l,r}$ of $r$ left singular vectors of $\bm{B}^{\alpha}_{\setminus l}$.
			\ENDFOR
			\STATE Set
			\begin{equation}
			\mathcal{E}(r)^2 = \frac{\sum_{l=1}^{z_{\alpha^c}} \Vert \bm{B}^{\alpha}_l - \bm{V}^{\alpha}_{\setminus l,r}(\bm{V}^{\alpha}_{\setminus l,r})^T \bm{B}^{\alpha}_l \Vert_2^2 }{\sum_{l=1}^{z_{\alpha^c}} \Vert \bm{B}^{\alpha}_l \Vert_2^2 }
			\end{equation}
			\ENDWHILE
			\ENDWHILE
			\STATE Set $r_{\alpha} = \min\{1 \le k \le r : \mathcal{E}(k) \le \varepsilon_c \}$ 
		\end{algorithmic}
	\end{algorithm}
	\section{Proofs}
	\label{appendix:proofs}
	\subsection{Proof of the Lemma \ref{lem:bound_bls_projection}}
	\label{appendix:proof_lemma21}		
	\begin{proof}
		First, let us show that the assumption \cref{eq:s-BLS_qo_projection-random} implies that for all $u \in L^2_{\mu}$, $\mathbb{E}(\Vert u - \mathcal{Q}_{V_{\alpha}}u \Vert^2) \le \left(1 + \gamma \right) \mathbb{E}(\Vert u - \mathcal{P}_{V_{\alpha}}u \Vert^2)$, with $\gamma = p_r(1-\delta)^{-1}(1-\eta^M)^{-1}M $.\\
		The function $u$ has a representation
		$u(x) = \sum_{k=1}^{\text{rank}_{\alpha}(u)}u_k^{\alpha}(x_{\alpha})u_k^{\alpha^c}(x_{\alpha^c})$
		with $\{u_k^{\alpha}\}$ an orthogonal family of functions.\\
		Then
		\begin{equation*}
		\begin{aligned}
		\Vert u - \mathcal{Q}_{V_{\alpha}}u \Vert^2 & = \Vert \sum_{k=1}^{\text{rank}_{\alpha}(u)}u_k^{\alpha} \otimes u_k^{\alpha^c} - \mathcal{Q}_{V_{\alpha}} \left( \sum_{k=1}^{\text{rank}_{\alpha}(u)}u_k^{\alpha}\otimes u_k^{\alpha^c} \right) \Vert^2\\
		& = \sum_{k=1}^{\text{rank}_{\alpha}(u)} \Vert( u_k^{\alpha} - Q_{V_{\alpha}} 
		u_k^{\alpha}) \otimes u_k^{\alpha^c} \Vert^2 \\
		& = \sum_{k=1}^{\text{rank}_{\alpha}(u)} \Vert u_k^{\alpha} - Q_{V_{\alpha}} 
		u_k^{\alpha} \Vert^2_{L^2_{\mu_{\alpha}}} \Vert u_k^{\alpha^c}  \Vert^2_{L^2_{\mu_{\alpha^c}}}. \\
		\end{aligned}
		\end{equation*}
		When $\text{rank}_{\alpha}(u) = \infty$, the series $\sum_{k=1}^{\text{rank}_{\alpha}(u)}u_k^{\alpha}\otimes u_k^{\alpha^c} - \mathcal{Q}_{V_{\alpha}}(u_k^{\alpha}\otimes u_k^{\alpha^c})$ is convergent by definition of $u$.\\
		By hypothesis on projection $Q_{V_{\alpha}}$ we have $\mathbb{E}(\Vert u_k^{\alpha} - Q_{V_{\alpha}} u_k^{\alpha} \Vert^2) \le \left(1 + \gamma \right) \mathbb{E}(\Vert u_k^{\alpha} - P_{V_{\alpha}} u_k^{\alpha} \Vert^2)$. Then 
		\begin{equation*}
		\begin{aligned}
		\mathbb{E}(\Vert u - \mathcal{Q}_{V_{\alpha}}u \Vert^2)  & \le \sum_{k=1}^{\text{rank}_{\alpha}(u)} \left(1 + \gamma \right) \mathbb{E}(\Vert u_k^{\alpha} - P_{V_{\alpha}} 
		u_k^{\alpha} \Vert^2) \Vert u_k^{\alpha^c}  \Vert^2 \\
		& = \sum_{k=1}^{\text{rank}_{\alpha}(u)} \left(1 + \gamma \right) \mathbb{E}(\Vert u_k^{\alpha} \otimes u_k^{\alpha^c} - (P_{V_{\alpha}} 
		u_k^{\alpha}) \otimes u_k^{\alpha^c} \Vert^2)\\
		& = \left(1 + \gamma \right)\mathbb{E}( \Vert u - \mathcal{P}_{V_{\alpha}} u \Vert^2).
		\end{aligned}
		\end{equation*}
		Now, thanks to the Pythagorean equality, we have $
		\mathbb{E}(\Vert u - \mathcal{Q}_{V_{\alpha}} u \Vert^2) = \mathbb{E}(\Vert u - \mathcal{P}_{V_{\alpha}} u \Vert^2) + \mathbb{E}(\Vert \mathcal{Q}_{V_{\alpha}}u - \mathcal{P}_{V_{\alpha}}u \Vert^2 )$, and then 
		\begin{equation*}
		\begin{aligned}
		\mathbb{E}(\Vert u - \mathcal{P}_{V_{\alpha}}u \Vert^2) + \mathbb{E}(\Vert \mathcal{Q}_{V_{\alpha}}u - \mathcal{P}_{V_{\alpha}}u \Vert^2)  & \le \left(1 + \gamma \right)\mathbb{E}( \Vert u - \mathcal{P}_{V_{\alpha}}u \Vert^2), \text{which implies } \\ \mathbb{E}(\Vert \mathcal{Q}_{V_{\alpha}}u - \mathcal{P}_{V_{\alpha}}u \Vert^2)  & \le \gamma \mathbb{E}(\Vert u - \mathcal{P}_{V_{\alpha}}u \Vert^2).
		\end{aligned}
		\end{equation*}
		Using the triangular inequality 
		$ \Vert \mathcal{Q}_{V_{\alpha}} u \Vert^2 \le 2\Vert \mathcal{Q}_{V_{\alpha}}u - \mathcal{P}_{V_{\alpha}}u \Vert^2 +  2\Vert \mathcal{P}_{V_{\alpha}}u \Vert^2$,
		we get
		\begin{equation*}
		\mathbb{E}(\Vert \mathcal{Q}_{V_{\alpha}} u \Vert^2) \le  2 \gamma \mathbb{E}(\Vert u - \mathcal{P}_{V_{\alpha}}u \Vert^2) + 2\mathbb{E}(\Vert \mathcal{P}_{V_{\alpha}} u \Vert^2) \le 2(\gamma +1)\Vert u \Vert^2,
		\end{equation*}
		which ends the proof.
	\end{proof}
	\subsection{Proof of the Theorem \ref{th:qo_constant_bls}}
	\begin{proof}
		By definition, for all $v \in L^2_{\mu}$ with $\text{rank}_{\alpha}(v) \le r_{\alpha}$, 
		\begin{equation*}
		\Vert \mathcal{Q}_{V_{\alpha}}u - \mathcal{P}_{U_{\alpha}^{\star}}\mathcal{Q}_{V_{\alpha}}u \Vert  =  \min_{\text{rank}_{\alpha}(v) \le r_{\alpha}} \Vert \mathcal{Q}_{V_{\alpha}}u - v\Vert.
		\end{equation*}
		If we choose in particular $v = \mathcal{Q}_{V_{\alpha}}\mathcal{P}_{U_{\alpha}}u$, where $U_{\alpha}$ is the $\alpha$-principal subspace of $u$, defined in Equation \Cref{eq:sol_principal_subs}, it comes
		\begin{equation*}
		\Vert \mathcal{Q}_{V_{\alpha}}u - \mathcal{P}_{U_{\alpha}^{\star}}\mathcal{Q}_{V_{\alpha}}u \Vert \le \Vert \mathcal{Q}_{V_{\alpha}}u - \mathcal{Q}_{V_{\alpha}}\mathcal{P}_{U_{\alpha}}u \Vert = \Vert \mathcal{Q}_{V_{\alpha}} (u - \mathcal{P}_{U_{\alpha}}u)\Vert.
		\end{equation*} 
		Taking the expectation and using Lemma \ref{lem:bound_bls_projection} it comes,
		\begin{equation*}
		\mathbb{E}(\Vert \mathcal{Q}_{V_{\alpha}}u - \mathcal{P}_{U_{\alpha}^{\star}}\mathcal{Q}_{V_{\alpha}}u \Vert^2) \le \mathbb{E}(\Vert \mathcal{Q}_{V_{\alpha}} (u - \mathcal{P}_{U_{\alpha}}u)\Vert^2) \le  C_1 e_{r_{\alpha}}^{\alpha}(u)^2.
		\end{equation*}
	\end{proof}
	\subsection{Proof of the Theorem \ref{th:final_err_bound}}
	\label{appendix:proof_theorem22}
	We start a preliminary results given by Lemma \ref{lem:ineq_proj_sons}, which is necessary for the proof of Lemma \ref{lem:bound_final_error_qo_proj}. Then, the Theorem \ref{th:final_err_bound} is deduced by making further assumptions on the reconstruction error of the empirical $\alpha$-principal subspace $\hat{U}_{\alpha}^{\star}$. 
	\begin{lemma}
		\label{lem:ineq_proj_sons}
		For $\alpha$ an interior node of a tree $T$ such that $V_{\alpha} = \bigotimes_{\beta \in S(\alpha)} \widehat{U}_{\beta}^{\star}$  we have
		\begin{equation*}
		\Vert u - \mathcal{P}_{V_{\alpha}}u \Vert^2 \le \sum_{\beta \in S(\alpha)} \Vert u - \mathcal{P}_{\widehat{U}_{\beta}^{\star}}u \Vert^2
		\end{equation*}
	\end{lemma}
	\begin{proof}
		Let $\gamma$ be an element of $S(\alpha)$, we have:
		\begin{equation*}
		\begin{aligned}
		\Vert u - \mathcal{P}_{V_{\alpha}}u \Vert^2 & = \Vert u - \prod_{\beta \in S(\alpha)} \mathcal{P}_{\widehat{U}_{\beta}^{\star}}u \Vert^2 \\
		& = \Vert u - \mathcal{P}_{\widehat{U}_{\gamma}^{\star}}u \Vert^2  + \Vert \mathcal{P}_{\widehat{U}_{\gamma}^{\star} }u - \mathcal{P}_{\widehat{U}_{\gamma}^{\star}} \prod_{\beta \in S(\alpha) \setminus \gamma} \mathcal{P}_{\widehat{U}_{\beta}^{\star}}u \Vert^2\\
		& \le  \Vert u - \mathcal{P}_{\widehat{U}_{\gamma}^{\star}}u \Vert^2  + \Vert 
		u - \prod_{\beta \in S(\alpha) \setminus \gamma} \mathcal{P}_{\widehat{U}_{\beta}^{\star}}u \Vert^2.
		\end{aligned}
		\end{equation*}
		Proceeding recursively, we obtain the desired result.
	\end{proof}
	\textbf{Proof of Lemma \ref{lem:bound_final_error_qo_proj}}.
	\begin{proof}
		The final approximation $u^{\star}$ is defined by $u^{\star} = \mathcal{Q}_{V_D}u$.\\
		For each $\alpha \in T$, thanks to the properties of $Q_{V_{\alpha}}$ we have from Lemma \ref{lem:bound_final_error_qo_proj}
		\begin{equation*}
		\mathbb{E}(\Vert u - \mathcal{Q}_{V_{\alpha}}
		u \Vert^2) \le C_1  \mathbb{E}(\Vert u - \mathcal{P}_{V_{\alpha}}u \Vert^2)
		\end{equation*}
		where $C_1$ is the constant associated to the boosted least-squares projection.\\
		If $\alpha \in \mathcal{L}(T)$, then $V_{\alpha}$ is a given deterministic space and $\mathbb{E}(\Vert u - \mathcal{P}_{V_{\alpha}} \Vert^2) =\Vert u - \mathcal{P}_{V_{\alpha}} \Vert^2$.\\
		If $\alpha \notin \mathcal{L}(T)$, then $V_{\alpha} = \bigotimes_{\beta \in S(\alpha) }\widehat{U}_{\beta}^{\star}$ and from Lemma \ref{lem:ineq_proj_sons},
		\begin{equation*}
		\label{eq:ineq_product_spaces_emp}
		\mathbb{E}(\Vert u - \mathcal{P}_{V_{\alpha}}u \Vert^2) \le \sum_{\beta \in S(\alpha)} \mathbb{E}(\Vert u - \mathcal{P}_{\widehat{U}_{\beta}^{\star}}u \Vert^2).
		\end{equation*}
		Using the triangular inequality, we can write
		\begin{equation*}\label{eq:ineq_emp_pca_spaces}
		\begin{aligned}
		\Vert u - \mathcal{P}_{\widehat{U}_{\beta}^{\star}}u \Vert & = \Vert u - \mathcal{Q}_{V_{\beta}}u + \mathcal{Q}_{V_{\beta}}u - \mathcal{P}_{\widehat{U}_{\beta}^{\star}}\mathcal{Q}_{V_{\beta}}u + \mathcal{P}_{\widehat{U}_{\beta}^{\star}}\mathcal{Q}_{V_{\beta}}u - \mathcal{P}_{\widehat{U}_{\beta}^{\star}}u \Vert\\
		& \le \Vert u - \mathcal{Q}_{V_{\beta}}u + \mathcal{P}_{\widehat{U}_{\beta}^{\star}}\mathcal{Q}_{V_{\beta}}u - \mathcal{P}_{\widehat{U}_{\beta}^{\star}}u \Vert + \Vert \mathcal{Q}_{V_{\beta}}u - \mathcal{P}_{\widehat{U}_{\beta}^{\star}}\mathcal{Q}_{V_{\beta}}u \Vert \\
		& \le \Vert (id - \mathcal{P}_{\widehat{U}_{\beta}^{\star}})(u - \mathcal{Q}_{V_{\beta}}u) \Vert + \Vert \mathcal{Q}_{V_{\beta}}u - \mathcal{P}_{\widehat{U}_{\beta}^{\star}}\mathcal{Q}_{V_{\beta}}u \Vert 
		\end{aligned}
		\end{equation*}
		so that,
		\begin{equation*}
		\Vert u - \mathcal{P}_{\widehat{U}_{\beta}^{\star}}u \Vert^2  \le 2\Vert u - \mathcal{Q}_{V_{\beta}}u\Vert^2 + 2\Vert \mathcal{Q}_{V_{\beta}}u - \mathcal{P}_{\widehat{U}_{\beta}^{\star}}\mathcal{Q}_{V_{\beta}}u \Vert^2
		\end{equation*}
		Using the equation (\ref{eq:ineq_product_spaces_emp}) and taking the expectation, it comes
		\begin{equation*}
		\mathbb{E}(\Vert u - \mathcal{P}_{V_{\alpha}}u \Vert^2) \le \sum_{\beta \in S(\alpha)} 2\mathbb{E}(\Vert u - \mathcal{Q}_{V_{\beta}}u\Vert^2) + 2\mathbb{E}(\Vert \mathcal{Q}_{V_{\beta}}u - \mathcal{P}_{\widehat{U}_{\beta}^{\star}}\mathcal{Q}_{V_{\beta}}u \Vert^2)
		\end{equation*}
		The term $\mathbb{E}(\Vert \mathcal{Q}_{V_{\beta}}u - \mathcal{P}_{\widehat{U}_{\beta}^{\star}}\mathcal{Q}_{V_{\beta}}u \Vert^2)$ is the error due to the principal component analysis. To deal with the term $\mathbb{E}(\Vert u - \mathcal{Q}_{V_{\beta}}u\Vert^2)$, we distinguish the case where $\beta$ is a leaf or not. If $\beta$ is not a leaf, we proceed recursively using \Cref{lem:ineq_proj_sons} and the triangular inequality. Going through all nodes, we obtain
		\begin{equation*}
		\label{eq:bound_final_error}
		\mathbb{E}(\Vert u - u^{\star} \Vert^2) \le \sum_{\alpha \in T \setminus D} (2C_1)^{l(\alpha)}\mathbb{E}(\Vert \mathcal{Q}_{V_{\alpha}}u - \mathcal{P}_{\widehat{U}_{\alpha}^{\star}}\mathcal{Q}_{V_{\alpha}}u \Vert^2) + \sum_{\alpha \in \mathcal{L}(T)} \frac{1}{2} (2C_1)^{l(\alpha)+1} \Vert u - \mathcal{P}_{V_{\alpha}}u \Vert^2.
		\end{equation*}
	\end{proof}
	The Theorem \ref{th:final_err_bound} is deduced by making further assumptions on the reconstruction error of the empirical $\alpha$-principal subspace $\hat{U}_{\alpha}^{\star}$. More precisely, we assume that we have for all $\alpha \in T \setminus \{D\}$
	\begin{equation}
	\label{eq:hyp_ineq_approx_empirical_subspaces_appendix}
	\mathbb{E}(\Vert \mathcal{Q}_{V_{\alpha}}u - \mathcal{P}_{\widehat{U}_{\alpha}^{\star}}\mathcal{Q}_{V_{\alpha}}u \Vert^2 \vert \mathcal{Q}_{V_{\alpha}}u) \le C_2 \mathbb{E}(\Vert \mathcal{Q}_{V_{\alpha}}u - \mathcal{P}_{U_{\alpha}^{\star}}\mathcal{Q}_{V_{\alpha}}u \Vert^2 \vert \mathcal{Q}_{V_{\alpha}}u),
	\end{equation}
	where $\Vert \mathcal{Q}_{V_{\alpha}}u - \mathcal{P}_{U_{\alpha}^{\star}}\mathcal{Q}_{V_{\alpha}}u \Vert^2$ is the reconstruction error associated with the $\alpha$-principal subspace of $U_{\alpha}^{\star}$ solution of Equation \Cref{eq:sol_approx_principal_subs}.
	\begin{proof}
		Taking the expectation in \Cref{eq:hyp_ineq_approx_empirical_subspaces_appendix}, we have for all $\alpha \in T \setminus \{D\}$.
		\begin{equation*}
		\label{eq:ineq_empirical_bound_int_nodes}
		\mathbb{E}(\mathbb{E}(\Vert \mathcal{Q}_{V_{\alpha}}u - \mathcal{P}_{\widehat{U}_{\alpha}^{\star}}\mathcal{Q}_{V_{\alpha}}u \Vert^2 \vert \mathcal{Q}_{V_{\alpha}}u)) \le C_2\mathbb{E}( \mathbb{E}(\Vert \mathcal{Q}_{V_{\alpha}}u - \mathcal{P}_{U_{\alpha}^{\star}}\mathcal{Q}_{V_{\alpha}}u \Vert^2 \vert \mathcal{Q}_{V_{\alpha}}u)),
		\end{equation*}
		which yields
		\begin{equation*}
		\mathbb{E}(\Vert \mathcal{Q}_{V_{\alpha}}u - \mathcal{P}_{\widehat{U}_{\alpha}^{\star}}\mathcal{Q}_{V_{\alpha}}u \Vert^2) \le C_2 \mathbb{E}(\Vert \mathcal{Q}_{V_{\alpha}}u - \mathcal{P}_{U_{\alpha}^{\star}}\mathcal{Q}_{V_{\alpha}}u \Vert^2).
		\end{equation*}
		Besides, the term $\mathbb{E}(\Vert \mathcal{Q}_{V_{\alpha}}u - \mathcal{P}_{U_{\alpha}^{\star}}\mathcal{Q}_{V_{\alpha}}u \Vert^2)$ can be bounded thanks to  Theorem \ref{th:qo_constant_bls}, such that
		\begin{equation*}
		\mathbb{E}(\Vert \mathcal{Q}_{V_{\alpha}}u - \mathcal{P}_{\widehat{U}_{\alpha}^{\star}}\mathcal{Q}_{V_{\alpha}}u \Vert^2) \le C_2 C_1 e_{r_{\alpha}}^{\alpha}(u)^2.
		\end{equation*}
		Using this bound and theorem \ref{lem:bound_final_error_qo_proj}, it comes
		\begin{equation*}
		\mathbb{E}(\Vert u - u^{\star} \Vert^2) \le C_1 C_2 \sum_{\alpha \in T \setminus D} (2C_1)^{l(\alpha)} e_{r_{\alpha}}^{\alpha}(u)^2 + \sum_{\alpha \in \mathcal{L}(T)} \frac{1}{2} (2C_1)^{l(\alpha)+1} e_{m_{\alpha}}^{\alpha,dis}(u)^2,
		\end{equation*}
		which ends the proof.
	\end{proof}

\end{appendix}

\end{document}